\documentclass[12pt]{amsart}

\usepackage[cp1251]{inputenc}
\usepackage{amsthm,amssymb,amsmath,amsfonts}
\usepackage{graphicx}
\usepackage[matrix,arrow,curve]{xy}
\usepackage{longtable}
\newcommand{\Z}{\ensuremath{\mathbb{Z}}}

\newcommand{\A}{\ensuremath{\mathbb{A}}}

\newcommand{\F}{\ensuremath{\mathbb{F}}}

\newcommand{\CG}{\ensuremath{\mathfrak{C}}}

\newcommand{\DG}{\ensuremath{\mathfrak{D}}}

\newcommand{\AG}{\ensuremath{\mathfrak{A}}}

\newcommand{\SG}{\ensuremath{\mathfrak{S}}}

\newcommand{\VG}{\ensuremath{\mathfrak{V}}}

\newcommand{\ka}{\ensuremath{\Bbbk}}

\newcommand{\kka}{\ensuremath{\overline{\Bbbk}}}

\newcommand{\XX}{\ensuremath{\overline{X}}}

\newcommand{\ii}{\ensuremath{\mathrm{i}}}

\newcommand{\Pro}{\ensuremath{\mathbb{P}}}

\newcommand{\Aut}{\ensuremath{\operatorname{Aut}}}

\newcommand{\Gal}{\ensuremath{\operatorname{Gal}}}

\newcommand{\Pic}{\ensuremath{\operatorname{Pic}}}

\newcommand{\ord}{\ensuremath{\operatorname{ord}}}

\makeatletter
\@addtoreset{equation}{section}
\makeatother

\newtheorem{theorem}[equation]{Theorem}
\newtheorem{proposition}[equation]{Proposition}
\newtheorem{lemma}[equation]{Lemma}
\newtheorem{corollary}[equation]{Corollary}

\theoremstyle{definition}
\newtheorem{example}[equation]{Example}
\newtheorem{definition}[equation]{Definition}

\theoremstyle{remark}
\newtheorem{remark}[equation]{Remark}
\newtheorem{notation}[equation]{Notation}

\title{Quotients of del Pezzo surfaces of degree 2}
\address{Institute for Information Transmission Problems, 19 Bolshoy Karetnyi side-str., Moscow 127994, Russia}
\address{Laboratory of Algebraic Geometry, National Research University Higher School of Economics, 6 Usacheva str., Moscow 119048, Russia}
\email{trepalin@mccme.ru}
\thanks{The author of article was supported by the Russian Academic Excellence Project '5--100', Young Russian Mathematics award, and the grant RFFI 15-01-02164-a}
\author{Andrey Trepalin}

\begin{document}

\begin{abstract}
Let $\ka$ be any field of characteristic zero, $X$ be a del Pezzo surface of degree~$2$ and $G$ be a group acting on $X$. In this paper we study $\ka$-rationality questions for the quotient surface $X / G$. If there are no smooth $\ka$-points on $X / G$ then $X / G$ is obviously non-$\ka$-rational.

Assume that the set of smooth $\ka$-points on the quotient is not empty. We find a list of groups such that the quotient surface can be non-$\ka$-rational. For these groups we construct examples of both $\ka$-rational and non-$\ka$-rational quotients \mbox{of both $\ka$-rational} and non-$\ka$-rational del Pezzo surfaces of degree $2$ such that the $G$-invariant Picard number of $X$ is~$1$. For all other groups we show that the quotient~$X / G$ is always $\ka$-rational.
\end{abstract}

\keywords{Rationality problems, del Pezzo surfaces, Minimal model program, Cremona group}

\maketitle
\tableofcontents

\section{Introduction}

In this paper we study quotients of del Pezzo surfaces of degree two by finite groups of automorphisms over arbitrary field $\ka$ of characteristic $0$. A surface $S$ is called \textit{$\ka$-rational} if there exists a birational map $\Pro^2_{\ka} \dashrightarrow S$ defined over $\ka$. If for the algebraic closure $\kka$ of~$\ka$ such a map defined over $\kka$ exists for a surface $\overline{S} = S \otimes_{\ka} \kka$ and $\Pro^2_{\kka}$, we say that $S$ is \textit{rational}. Note that in many other papers for these notions the authors use terms \textit{rational surface} and \textit{geometrically rational surface} respectively. A smooth surface $S$ is called \textit{minimal} if any birational morphism of smooth surfaces $S \rightarrow S'$ is an isomorphism. The following theorem is an important criterion of $\ka$-rationality over an arbitrary perfect field $\ka$.

\begin{theorem}[{\cite[Chapter 4]{Isk96}}]
\label{ratcrit}
A minimal rational surface $X$ over a perfect field $\ka$ is $\ka$-rational if and only if the following two conditions are satisfied:

(i) $X(\ka) \neq \varnothing$;

(ii) $K_X^2 \geqslant 5$.
\end{theorem}

If the field $\ka$ is algebraically closed then the quotient of any $\ka$-rational surface by any finite group $G$ is $\ka$-rational by Castelnuovo's criterion. But if $\ka$ is not algebraically closed, then for a finite geometric group $G \subset \Aut_{\ka}(X)$ the quotient $X / G$ can be non-$\ka$-rational. For example, by \cite[Theorem~IV.7.8]{Man74} any del Pezzo surface $X$ of degree $4$ such that $X(\ka) \neq \varnothing$, is $\ka$-unirational of degree~$2$ (i.e. birationally equivalent to a quotient of a $\ka$-rational surface by an involution), but a minimal del Pezzo surface of degree $4$ is not $\ka$-rational by Theorem \ref{ratcrit}.

By \cite[Theorem 1]{Isk79} any quotient of $\ka$-rational surface by a finite group $G$ of automorphisms of this surface is birationally equivalent either to a quotient $X / G$ of a $G$-equivariant conic bundle $X \rightarrow \Pro^1_{\ka}$ such that~$\rho(X)^G = 2$; or to a quotient $X / G$ of a del Pezzo surface~$X$ such that $\rho(X)^G = 1$. Quotients of conic bundles were considered in \cite{Tr16a}. Quotients of del Pezzo surfaces of degree $4$ and higher were considered in \cite{Tr17}.

\begin{theorem}[{\cite[Theorem 1.1]{Tr17}}]
\label{DP4more}
Let $\ka$ be a field of characteristic zero, $X$ be a del Pezzo surface over $\ka$ such that $X(\ka) \ne \varnothing$, and $G$ be a finite subgroup of $\Aut_{\ka}(X)$. If $K_X^2 \geqslant 5$ then the quotient variety $X / G$ is $\ka$-rational.

If $K_X^2 = 4$, and the order of $G$ is not equal to $1$, $2$ and $4$, or there is a nontrivial element in $G$ that has a curve of fixed points, then $X / G$ is $\ka$-rational. There exists an example of a non-$\ka$-rational quotient $X / G$ for a suitable field $\ka$, if the order of $G$ is equal to~$1$, $2$, or~$4$, and all nontrivial elements of $G$ have only isolated fixed points.
\end{theorem}

Moreover we have the following corollary, that is a generalization of Theorem \ref{ratcrit}.

\begin{corollary}[{\cite[Corollary 1.2]{Tr17}}]
\label{geq5}
Let $\ka$ be a field of characteristic zero, $X$ be a smooth rational surface over $\ka$ such that $X(\ka) \ne \varnothing$, and $G$ be a finite subgroup of $\Aut_{\ka}(X)$. If $K_X^2 \geqslant 5$ then the quotient variety $X / G$ is $\ka$-rational.
\end{corollary}

\begin{remark}
\label{geq5touse}
Actually, in \cite{Tr16a} and \cite{Tr17} it is proved that if $X$ is a smooth rational surface over $\ka$ such that $K_X^2 \geqslant 5$ and $G$ is a finite subgroup of $\Aut_{\ka}(X)$ then the quotient~$X / G$ is birationally equivalent to a surface $Y$ such that $K_Y^2 \geqslant 5$.

\end{remark}

For del Pezzo surfaces of degree $5$ or less the automorphisms group is finite. Therefore in the following theorems the condition, that $G$ is finite, is not necessary.

Quotients of del Pezzo surfaces of degree $3$ were considered in \cite{Tr16b}.

\begin{theorem}[{\cite[Theorem 1.3]{Tr16b}}]
\label{DP3}
Let $\ka$ be a field of characteristic zero, $X$ be a del Pezzo surface over $\ka$ of degree $3$ such that $X(\ka) \neq \varnothing$, and $G$ be a subgroup of $\Aut_{\ka}(X)$. If the order of $G$ is not equal to $1$ or $3$, or there is a nontrivial element in $G$ that has a curve of fixed points, then $X / G$ is $\ka$-rational. There exists an example of a non-$\ka$-rational quotient $X / G$ for a suitable field $\ka$, if the order of $G$ is equal to $1$ or $3$, and all nontrivial elements of $G$ do not have curves of fixed points.
\end{theorem}

In this paper we want to find and classify finite groups~$G$ such that the quotient $X / G$ of a del Pezzo surface $X$ of degree $2$ is not $\ka$-rational. From Theorems \ref{DP4more} and~\ref{DP3} one can get an impression that if $X(\ka) \neq \varnothing$, and the group $G$ contains an element $g$ such that the set of fixed points of $g$ contains a curve, then $X / G$ is going to be $\ka$-rational. We show that this is not always true for del Pezzo surfaces of degree $2$. Moreover, almost all groups such that the quotient $X / G$ can be non-$\ka$-rational, contain an involution that has a curve of fixed points. The main result of this paper is the following.

\begin{theorem}
\label{DP2briefly}
Let $\ka$ be a field of characteristic zero, $X$ be a del Pezzo surface over $\ka$ of degree $2$, and $G$ be a subgroup of $\Aut_{\ka}(X)$ such that there is a smooth $\ka$-point on $X / G$. Then the quotient~$X / G$ is $\ka$-rational for any group $G$ except the following groups:
\begin{itemize}
\item a trivial group;
\item a group of order $2$;
\item a cyclic group of order $4$ containing a unique involution that has a curve of fixed points;
\item an abelian noncyclic group of order $4$ containing a unique involution that has a~curve of fixed points;
\item a dihedral group of order $8$ containing a unique involution that has a curve of fixed points;
\item a quaternion group of order $8$ containing a unique involution that has a curve of fixed points;
\item a group of order $3$ that has only isolated fixed points;
\item a symmetric group of degree $3$ generated by involutions that have only isolated fixed points.
\end{itemize}
For each of the latter groups there exists an example of a non-$\ka$-rational quotient $X / G$ for a suitable field $\ka$.
\end{theorem}

\begin{remark}
\label{density}
We need a smooth $\ka$-point on $X / G$, to apply Theorem \ref{ratcrit}. Without this condition for the groups not listed in Theorem \ref{DP2briefly} we can only say that the quotient~$X / G$ is birationally equivalent to a surface $Y$ such that $K_Y^2 \geqslant 5$.

Note that this condition, that there is a smooth $\ka$-point on $X / G$, differs from the corresponding condition $X(\ka) \neq \varnothing$, that is used in Theorems \ref{DP4more} and \ref{DP3}. For del Pezzo surfaces of degree $3$ or greater if $X(\ka) \neq \varnothing$ then $X$ is $\ka$-unirational (see \cite[Theorems~IV.7.8 and IV.8.1]{Man74}). In particular, $X(\ka)$ is dense. This fact immediately implies that the set of smooth $\ka$-points on $X / G$ is dense.

If $X$ is a non-minimal del Pezzo surface $X$ of degree $2$ and $X(\ka) \neq \varnothing$ then $X(\ka)$ is dense, since $X$ is birationally equivalent to a del Pezzo surface of higher degree, which is $\ka$-unirational. Moreover, if $X$ is minimal and there is a $\ka$-point that is not a point of intersection of four $(-1)$-curves and does not lie on the ramification divisor of the anticanonical map $X \rightarrow \Pro^2_{\ka}$, then $X$ is $\ka$-unirational by \cite[Corollary 18]{STV14} and $X(\ka)$ is dense. In all these cases the set of smooth $\ka$-points on $X / G$ is dense.

It seems that nobody knows an answer to the question about $\ka$-unirationality of arbitrary del Pezzo surface $X$ such that $X(\ka) \neq \varnothing$. If any del Pezzo surface of degree $2$ such that $X(\ka) \neq \varnothing$, is $\ka$-unirational then in Theorem \ref{DP2briefly} we can replace the assumption of existence of a smooth $\ka$-point on the quotient to the assumption $X(\ka) \neq \varnothing$. In any case, the assumption of existence of a smooth $\ka$-point on the quotient holds if the set of $\ka$-points on $X$ is dense, that follows from the assumption that there is a $\ka$-point on $X$ not lying on $(-1)$-curves and the ramification divisor of the anticanonical map $X \rightarrow \Pro^2_{\ka}$. 

Actually, in Lemma \ref{Quotpoints} for many groups listed in Theorem \ref{DP2briefly} we show that the condition $X(\ka) \neq \varnothing$ implies that the set of $\ka$-points on $X / G$ is dense.


\end{remark}

One can find a more precise statement of Theorem \ref{DP2briefly} in Proposition \ref{DP2}.

Note that a minimal del Pezzo surface $X$ of degree $2$ such that $X(\ka) \neq \varnothing$ is not \mbox{$\ka$-rational} by Theorem \ref{ratcrit}. Thus the quotient of~$X$ by the trivial group is not $\ka$-rational in this case. Let $G$ be a nontrivial group listed in Theorem~\ref{DP2briefly}. If $\rho(X)^G > 1$ then either $X$ admits a structure of a $G$-equivariant conic bundle, or there is a $G$-equivariant morphism $X \rightarrow Y$, where $Y$ is a del Pezzo surface of degree greater than $2$. Quotients of such surfaces were considered in \cite{Tr16a}, \cite{Tr17} and~\cite{Tr16b}. Therefore we assume that~$\rho(X)^G = 1$. We show that if the group $G$ is of order $2$ without curves of fixed points then $X$ and $X / G$ are not \mbox{$\ka$-rational}. For each of the other possibilities for $G$ we construct examples, for which all the following options occur: $X$ is \mbox{$\ka$-rational} and $X / G$ is non-$\ka$-rational, $X$ is non-$\ka$-rational and $X / G$ is non-$\ka$-rational, $X$~is $\ka$-rational and $X / G$ is $\ka$-rational, $X$~is non-$\ka$-rational and $X / G$ is $\ka$-rational.

The plan of this paper is as follows. In Section $2$ we remind some notions and results about rational surfaces, $G$-equivariant minimal model program, groups, singularities and conic bundles.

In Section $3$ we consider a special conic bundle with $4$ singular fibres called \textit{Iskovskikh surface}, and study quotients of this surface. We are interested in quotients of this surface, since the quotient of a del Pezzo surface by a group of order $2$ with a curve of fixed points is birationally equivalent to an Iskovskikh surface.

In Section $4$ we consider quotients of del Pezzo surfaces of degree $2$. In Proposition~\ref{DP2} we give a list of groups such that the quotient can be non-$\ka$-rational. To prove this proposition we consider groups acting on del Pezzo surfaces of degree $2$, and show that the quotient is $\ka$-rational in all cases that are not listed in Proposition \ref{DP2}. Also in Remark~\ref{DP2type2min} we show that the quotient $X / G$ of a del Pezzo surface $X$ of degree $2$ by a group~$G$ of order $2$ without curves of fixed points is always not $\ka$-rational if~$\rho(X)^G = 1$.

In Section $5$ we consider some properties of the Weyl group $\mathrm{W}(\mathrm{E}_7)$. For any group $G$ acting on a del Pezzo surface $X$ of degree $2$ there exists an embedding $G \hookrightarrow \mathrm{W}(\mathrm{E}_7)$. Moreover, the image of this embedding commutes with the image of the Galois group~$\Gal\left( \kka / \ka \right)$ in $\mathrm{W}(\mathrm{E}_7)$. In Lemma \ref{DP2type2nonrat} we show that if a group~$G$ of order $2$ without curves of fixed points acts on a del Pezzo surface $X$ of degree $2$ and $\rho(X)^G = 1$, then $X$ is not $\ka$-rational. Also we find the centralizer in $\mathrm{W}(\mathrm{E}_7)$ of the image of a group of order~$3$ acting on a del Pezzo surface of degree $2$ without curves of fixed points. Later we use this result to construct examples of quotients of del Pezzo surfaces of degree $2$ by groups of order $3$ and~$6$.

In Section $6$ we consider groups $G$ listed in Proposition \ref{DP2}, that are neither trivial, nor groups of order $2$ that have only isolated fixed points. For each of these groups we construct explicit examples of $\ka$-rational and non-$\ka$-rational quotients of $\ka$-rational and non-$\ka$-rational del Pezzo surfaces~$X$ of degree $2$ such that $\rho(X)^G = 1$.

The author is grateful to Costya Shramov for many useful discussions and comments, and to Yuri\,G.\,Prokhorov and Egor Yasinsky for valuable comments.
Also the author would like to thank the reviewer of this paper for many useful comments.

\begin{notation}

Throughout this paper $\ka$ is any field of characteristic zero, $\kka$ is its algebraic closure. For a surface $X$ we denote $X \otimes \kka$ by $\XX$. For a surface $X$ we denote the Picard group (resp. the $G$-invariant Picard group) by $\Pic(X)$ (resp. $\Pic(X)^G$). The number \mbox{$\rho(X) = \operatorname{rk} \Pic(X)$} (resp. \mbox{$\rho(X)^G = \operatorname{rk} \Pic(X)^G$}) is the Picard number (resp. the $G$-invariant Picard number) of $X$. If two surfaces $X$ and $Y$ are $\ka$-birationally equivalent then we write~$X \approx Y$. If two divisors $A$ and $B$ are linearly equivalent then we write~$A \sim B$.

\end{notation}

\section{Preliminaries}

\subsection{$G$-minimal rational surfaces}

In this subsection we review main notions and results of $G$-equivariant minimal model program following the papers \cite{Man67}, \cite{Isk79}, \cite{DI1}. Throughout this paper $G$ is a finite group.

\begin{definition}
\label{rationality}
A {\it rational variety} $X$ is a variety over $\ka$ such that $\XX=X \otimes \kka$ is birationally equivalent to $\Pro^n_{\kka}$.

A {\it $\ka$-rational variety} $X$ is a variety over $\ka$ such that $X$ is birationally equivalent to $\Pro^n_{\ka}$.

A variety $X$ over $\ka$ is a {\it $\ka$-unirational variety} if there exists a $\ka$-rational variety $Y$ and a dominant rational map $\varphi: Y \dashrightarrow X$.
\end{definition}

\begin{definition}
\label{minimality}
A {\it $G$-surface} is a pair $(X, G)$ where $X$ is a projective surface over $\ka$ and~$G$ is a finite subgroup of $\Aut_{\ka}(X)$. A morphism of $G$-surfaces $f: X \rightarrow X'$ is called a \textit{$G$-morphism} if for each $g \in G$ one has $fg = gf$.

A smooth $G$-surface $(X, G)$ is called {\it $G$-minimal} if any birational $G$-morphism of smooth $G$-surfaces $(X, G) \rightarrow (X',G)$ is an isomorphism.

Let $(X, G)$ be a smooth $G$-surface. A $G$-minimal surface $(Y, G)$ is called a {\it minimal model} of $(X, G)$ or {\it $G$-minimal model} of $X$ if there exists a birational \mbox{$G$-morphism $X \rightarrow Y$}.
\end{definition}

The following theorem is a classical result about the $G$-equivariant minimal model program.

\begin{theorem}
\label{GMMP}
Any birational $G$-morphism $f:X \rightarrow Y$ of smooth $G$-surfaces can be factorized in the following way:
$$
X= X_0 \xrightarrow{f_0} X_1 \xrightarrow{f_1} \ldots \xrightarrow{f_{n-2}} X_{n-1} \xrightarrow{f_{n-1}} X_n = Y,
$$
where each $f_i$ is a contraction of a set $\Sigma_i$ of disjoint $(-1)$-curves on~$X_i$ such that $\Sigma_i$ is defined over $\ka$ and $G$-invariant.
\end{theorem}

The classification of $G$-minimal rational surfaces is well-known due to V.\,Iskovskikh and Yu.\,Manin (see \cite{Isk79} and \cite{Man67}). We introduce some important notions before surveying it.

\begin{definition}
\label{Cbundledef}
A smooth rational $G$-surface $(X, G)$ admits a structure of a {\it $G$-equivariant conic bundle} if there exists a $G$-equivariant map $\varphi: X \rightarrow B$ such that any scheme fibre is isomorphic to a reduced conic in~$\Pro^2_{\ka}$ and $B$ is a smooth curve.

The curve $B$ is called the {\it base} of the conic bundle.

\end{definition}

Let $\overline{\varphi}: \XX \rightarrow \overline{B} \cong \Pro^1_{\kka}$ be a conic bundle. A general fibre of $\overline{\varphi}$ is isomorphic to $\Pro^1_{\kka}$. The fibration~$\overline{\varphi}$ has a finite number of singular fibres which are degenerate conics. Any irreducible component of a singular fibre is a $(-1)$-curve. If $n$ is the number of singular fibres of $\overline{\varphi}$ then $K_{\XX}^2 + n = 8$.

\begin{definition}
\label{relmin}
Let $X$ be a $G$-surface that admits a conic bundle structure \mbox{$\varphi: X \rightarrow B$}. The conic bundle is called \textit{relatively $G$-minimal} over $B$ if for any decomposition of $\varphi$ into $G$-morphisms
$$
X \xrightarrow{\psi} X' \rightarrow B
$$
\noindent such that the morphism $\psi$ is birational, the morphism $\psi$ is an isomorphism.

Let $X$ be a $G$-surface that admits a conic bundle structure \mbox{$\varphi: X \rightarrow B$}. A relatively $G$-minimal surface $\varphi': Y \rightarrow B$ is called a {\it relatively $G$-minimal model of $X$ over $B$}, if there exists a birational $G$-morphism $\psi: X \rightarrow Y$ such that $\varphi = \varphi' \circ \psi$.

\end{definition}

A conic bundle \mbox{$\varphi: X \rightarrow B$} is relatively $G$-minimal over $B$ if and only if $\rho(X)^G = 2$.

\begin{definition}
\label{DPdef}
A {\it del Pezzo surface} is a smooth projective surface~$X$ such that the anticanonical class $-K_X$ is ample.

A {\it singular del Pezzo surface} is a normal projective surface $X$ such that the anticanonical class $-K_X$ is ample and all singularities of $X$ are Du Val singularities.

A {\it weak del Pezzo surface} is a smooth projective surface $X$ such that the anticanonical class $-K_X$ is nef and big.

The number $d = K_X^2$ is called the {\it degree} of a (singular, weak) del Pezzo surface $X$.
\end{definition}

The following proposition is well known (see e.g. \cite[Subsection 8.1.3]{Dol12}). 

\begin{proposition}
\label{DPconnection}
If $X$ is a singular del Pezzo surface and $\widetilde{X} \rightarrow X$ is the minimal resolution of singularities then $X$ is a weak del Pezzo surface.

If $\widetilde{X}$ is a weak del Pezzo surface and $\widetilde{X} \rightarrow Y$ is a birational morphism of smooth surfaces then $Y$ is a weak del Pezzo surface.

If $Y$ is a weak del Pezzo surface and there are no $(-2)$-curves on $Y$ then $Y$ is a del Pezzo surface.

\end{proposition}

A del Pezzo surface $\XX$ over $\kka$ is isomorphic to $\Pro^2_{\kka}$, $\Pro^1_{\kka} \times \Pro^1_{\kka}$ or the blowup of $\Pro^2_{\kka}$ at up to 8 points in general position (see \cite[Theorem 2.5]{Man74}).




\begin{theorem}[{\cite[Theorem 1]{Isk79}}]
\label{Minclass}
Let $X$ be a $G$-minimal rational $G$-surface. Then either $X$ admits a $G$-equivariant conic bundle structure with $\Pic(X)^{G} \cong \Z^2$, or $X$ is a del Pezzo surface with $\Pic(X)^{G} \cong \Z$.
\end{theorem}

\begin{theorem}[{cf. \cite[Theorems 4 and 5]{Isk79}}]
\label{MinCB}
Let $X$ admit a $G$-equivariant structure of a conic bundle such that $\rho(X)^G = 2$ and $K_X^2 \neq 3$, $5$, $6$, $7$, $8$. Then $X$ is $G$-minimal.
\end{theorem}

\subsection{Groups}

In this paper we use the following notation:

\begin{itemize}

\item $\ii = \sqrt{-1}$;

\item $\xi_k = e^{\frac{2\pi \ii}{k}}$;

\item $\omega = \xi_3 = e^{\frac{2\pi \ii}{3}}$;

\item $\CG_n$ denotes a cyclic group of order $n$;

\item $\DG_{2n}$ denotes a dihedral group of order $2n$;

\item $\SG_n$ denotes a symmetric group of degree $n$;

\item $\AG_n$ denotes an alternating group of degree $n$;

\item $(i_1 i_2 \ldots i_j)$ denotes a cyclic permutation of $i_1$, \ldots, $i_j$;

\item $\VG_4$ denotes a Klein group isomorphic to $\CG_2^2$;

\item $Q_8$ denotes a quaternion group of order $8$;

\item $\langle g_1, \ldots, g_n \rangle$ denotes a group generated by $g_1$, \ldots, $g_n$;

\item $A \times B$ denotes the direct product of groups $A$ and $B$;

\item $A \rtimes B$ denotes a semi-direct product of groups $A$ and $B$, defined by a homomorphism $B \rightarrow \Aut(A)$;

\item $\operatorname{diag}(a_1, \ldots, a_n)$ denotes the diagonal $n \times n$ matrix with diagonal entries $a_1$, \ldots $a_n$;

\item for a vector space $V$ (or a lattice $L$) with an action of a group $G$ we denote by $V^G$ (resp. $L^G$) the subspace of $G$-invariant vectors (resp. the sublattice of $G$-invariant elements).

\end{itemize}

\subsection{Singularities}

In this subsection we review some results about quotient singularities and their resolutions.

All singularities appearing in this paper are toric singularities. These singularities are locally isomorphic to the quotient of $\A^2$ by the cyclic group generated by $\operatorname{diag}(\xi_m, \xi_m^q)$. Such a singularity is denoted by $\frac{1}{m}(1,q)$. If $\gcd(m,q) > 1$ then the group
$$
\CG_m \cong \langle \operatorname{diag}(\xi_m, \xi_m^q) \rangle
$$
\noindent contains a reflection (i.e. an element with a unique eigenvalue not equal to $1$) and the quotient singularity is isomorphic to a quotient singularity with smaller $m$. A singularity of type $\frac{1}{m}(1, m-1)$ is called \textit{$A_{m-1}$-singularity}.

A toric singularity can be resolved by a sequence of weighted blowups. Therefore it is easy to describe numerical properties of a quotient singularity. We list here these properties for singularities appearing in this paper.

Let the group $\CG_m$ act on a smooth surface $X$ and \mbox{$f: X \rightarrow S$} be the quotient map. Let $p$ be a singular point on $S$ of type $\frac{1}{m}(1,q)$. Let $C$ and $D$ be curves passing through $p$ such that $f^{-1}(C)$ and $f^{-1}(D)$ are $\CG_m$-invariant and tangent vectors of these curves at the point $f^{-1}(p)$ are eigenvectors of the natural action of $\CG_m$ on the Zariski tangent space~$T_{f^{-1}(p)} X$ (the curve $C$~corresponds to the eigenvalue $\xi_m$ and the curve $D$ corresponds to the eigenvalue~$\xi_m^q$).

Let $\pi: \widetilde{S} \rightarrow S$ be the minimal resolution of the singular point $p$. Table \ref{table1} presents some numerical properties of $\widetilde{S}$ and $S$ for the singularities appearing in this paper.

The exceptional divisor of $\pi$ is a chain of transversally intersecting exceptional curves~$E_i$ whose self-intersection numbers are listed in the last column of Table \ref{table1}.
The curve~$\pi^{-1}_*(C)$ transversally intersects at a point the curve $E_1$, and $\pi^{-1}_*(D)$ transversally intersects at a point only the last curve in the chain. The curves $\pi^{-1}_*(C)$ and $\pi^{-1}_*(D)$ do not intersect other components of exceptional divisor of~$\pi$.


\begin{table}
\caption{} \label{table1}

\begin{tabular}{|c|c|c|c|c|c|}
\hline
$m$ & $q$ & $K_{\widetilde{S}}^2 - K_S^2$ & $\pi^{-1}_*(C)^2 - C^2$ & $\pi^{-1}_*(D)^2 - D^2$ \rule[-7pt]{0pt}{20pt} & $E_i^2$ \\
\hline
$n$ & $n-1$ & $0$ & $-\dfrac{n-1}{n}$ & $-\dfrac{n-1}{n}$ \rule[-11pt]{0pt}{30pt} & $-2$, \ldots, $-2$ ($n-1$ times) \\
\hline
$3$ & $1$ & $-\dfrac{1}{3}$ & $-\dfrac{1}{3}$ & $-\dfrac{1}{3}$ \rule[-11pt]{0pt}{30pt} & $-3$ \\
\hline
$7$ & $3$ & $-\dfrac{3}{7}$ & $-\dfrac{3}{7}$ & $-\dfrac{5}{7}$ \rule[-11pt]{0pt}{30pt} & $-3$, $-2$, $-2$ \\
\hline
\end{tabular}

\end{table}

\subsection{Conic bundles}

In this subsection we collect some facts that will be used in Section~$3$ to work with conic bundles.

The following two lemmas are well known technical facts.

\begin{lemma}[{see \cite[Lemma 3.2]{Tr16a}}]
\label{fixedpoints}
Let elements $g_1, g_2 \in \mathrm{PGL}_2 \left( \kka \right)$ generate a finite group~$H = \langle g_1, g_2\rangle$. Then the elements $g_1$ and $g_2$ have the same pair of fixed points on $\Pro^1_{\kka}$ if and only if the group $H$ is cyclic. If $g_1$ and $g_2$ have a common fixed point on $\Pro^1_{\kka}$ then they have the same pair of fixed points.
\end{lemma}

\begin{lemma}[{see \cite[Lemma 4.5]{Tr16a}}]
\label{blowup}
Let $p$ be a smooth point of a surface $S$ and~$g$ be an element of $\Aut(S)$ that fixes the point $p$. Let $g$ act on $T_p S$ as $\operatorname{diag}\left(\lambda, \mu \right)$ and $\pi: \widetilde{S} \rightarrow S$ be the blowup of $S$ at the point $p$. Then $g$ has exactly two fixed points $p_1$ and $p_2$ on the exceptional divisor of $\pi$ and acts on $T_{p_1}\widetilde{S}$ and $T_{p_2}\widetilde{S}$ as $\operatorname{diag}\left(\frac{\lambda}{\mu}, \mu \right)$ and~$\operatorname{diag}\left(\lambda, \frac{\mu}{\lambda} \right)$ respectively.
\end{lemma}

We need the following two facts to consider quotients of conic bundles.

\begin{lemma}[{see \cite[Lemma 4.6]{Tr16a}}]
\label{eveninvariantfibre}
Let $\varphi: X \rightarrow B$ be a $G$-minimal conic bundle, and~$g \in G$ be an element of even order. Then $g$-invariant fibres of $\varphi$ are smooth.
\end{lemma}

\begin{theorem}[{cf. \cite[Theorem 4.1]{Tr16a}}]
\label{Cbundle}
Let $X$ be a $G$-surface that admits a \mbox{$G$-equivariant} conic bundle structure \mbox{$\varphi: X \rightarrow B$} such that $\overline{B} \cong \Pro^1_{\kka}$, $K_X^2 = 4$, $X$ is relatively \mbox{$G$-minimal} and the group $G$ faithfully acts on $B$. Let $\widetilde{X / G} \rightarrow X / G$ be the minimal resolution of singularities. Then for any relatively minimal model $Y$ of $\widetilde{X / G}$ over $B / G$ one has $K_Y^2 \geqslant 4$. Moreover, $K_Y^2 = 4$ only if $G \cong \CG_2$ or $G \cong \CG_2^2$.

\end{theorem}

\section{Iskovskikh surface}

In this section following \cite[Subsection 5.2]{DI1} we construct one specific conic bundle and study its quotients by finite groups.

\begin{definition}
\label{hyperellipticcurve}
Let an involution $\iota$ act on a smooth curve $C$ so that the quotient $C / \langle \iota \rangle$ is isomorphic to a smooth conic $B$. Then the curve $C$ is called \textit{a hyperelliptic curve}.

\end{definition}

One can easily check that $\iota$ has exactly $2g + 2$ fixed $\kka$-points on $C$ where $g$ is the genus of $C$.

Let $\iota$ nontrivially act on $\Pro^1_{\ka}$. Consider the quotient $\left( C \times \Pro^1_{\ka} \right) / \langle \iota \rangle$. It admits a structure of a $\Pro^1_{\ka}$-bundle over $B = C / \langle \iota \rangle$, and has $2g + 2$ fibres defined over $\kka$ each of which contains two singularities of type $A_1$. One can resolve these singularities and contract the proper transforms of the~$2g+2$~corresponding fibres.

We get a conic bundle $X \rightarrow B$ with $2g + 2$ singular fibres defined over~$\kka$. Note that all singularities of $\left( C \times \Pro^1_{\ka} \right) / \langle \iota \rangle$ lie in two disjoint sections of \mbox{the $\Pro^1_{\ka}$-bundle $\left( C \times \Pro^1_{\ka} \right) / \langle \iota \rangle \rightarrow B$}. Thus there are two sections of $X \rightarrow B$ with self-intersection number $-g - 1$.

\begin{definition}
\label{excCB}
The obtained conic bundle $X \rightarrow B$ is called \textit{an exceptional conic bundle} (see~\cite[Subsection 5.2]{DI1}). If $g = 1$ then we call $X$ \textit{an Iskovskikh surface}.
\end{definition}

\begin{remark}
Note that any conic bundle $X \rightarrow B$ with $2g + 2$ singular fibres and two sections with self-intersection number $- g - 1$ is exceptional, since one can reverse the construction above.
\end{remark}

One can check that over $\kka$ any exceptional conic bundle can be obtained as a blowup of $2g + 2$ points on~$\Pro^1_{\kka} \times \Pro^1_{\kka}$ lying in two disjoint sections of the projection $\Pro^1_{\kka} \times \Pro^1_{\kka} \rightarrow \Pro^1_{\kka}$ on the first factor. Thus any Iskovskikh surface is a weak del Pezzo surface. By contracting two sections with self-intersection number $-2$ on it one can get a singular del Pezzo surface of degree~$4$ with two singularities of type $A_1$.

\begin{remark}
\label{terminology}
In \cite[Section 7]{CT88} the term \textit{Iskovskikh surface} is used for the singular del Pezzo surface constructed above. This surface for the first time appeared in \cite{Isk71}. But in this paper we mainly focus on properties of the corresponding weak del Pezzo surface. We hope that Definition \ref{excCB} will not cause misunderstanding.
\end{remark}



We need the following result of the paper \cite{CT88}.

\begin{lemma}[{cf. \cite[Lemma 7.1]{CT88}}]
\label{IskUnirat}
If an Iskovskikh surface $X$ contains a $\ka$-point that does not lie on a section with self-intersection number $-2$, then~$X$ is $\ka$-unirational. In particular, $X(\ka)$ is dense.
\end{lemma}

We use the following notation in this section.

Let $X$ be an Iskovskikh surface and $\varphi: X \rightarrow B$ be the corresponding conic bundle. Let $p_1$, $p_2$, $p_3$ and $p_4$ be points on $\overline{B} \cong \Pro^1_{\kka}$ such that $\overline{\varphi}^{-1}(p_i) = E_i \cup E'_i$ are singular fibres. Let $C$ and $C'$ be the sections of $\varphi: X \rightarrow B$ with self-intersection number $-2$. One has $C \cdot E_i = 1$, $C \cdot E'_i = 0$, $C' \cdot E_i = 0$ and $C' \cdot E'_i = 1$ on $\XX$.

For any finite group $G$ acting on $X$ we denote by $G_0$ the normal subgroup acting trivially on the set of all $(-1)$-curves and by $G_B$ the subgroup acting trivially on the base~$B$. The group $G_0$ fixes the four points $p_i$ on $B$ thus it acts trivially on $B$ and $G_0 \subset G_B$.

We want to find possibilities for $G$ such that the quotient $X / G$ is not birationally equivalent to a surface $Y$ with $K_Y^2 \geqslant 5$. If $X$ is not $G$-minimal then it is birationally equivalent to a surface $S$ such that $K_S^2 \geqslant 5$. Thus $X / G \approx S / G$ is birationally equivalent to a surface $Y$ such that $K_Y^2 \geqslant 5$ by Remark \ref{geq5touse}. Therefore we assume that $X$ is \mbox{$G$-minimal}, i.\,e. $\rho(X)^G = 2$.

\begin{lemma}
\label{IskG0}
The group $G_0$ is cyclic. If $|G_0|$ is even then the quotient $X / G_0$ is \mbox{$G / G_0$-birationally} equivalent to a surface $Y$ such that $K_Y^2 = 8$. If $|G_0|$ is odd then the quotient $X / G_0$ is $G / G_0$-birationally equivalent to a $G / G_0$-minimal Iskovskikh surface.
\end{lemma}

\begin{proof}

The group $G_0$ acts trivially on $C$ and $C'$. Therefore $G_0$ faithfully acts on any smooth fibre $F$ of $\overline{\varphi}: \XX \rightarrow \overline{B}$ and fixes the points $F \cap C$ and $F \cap C'$. Thus $G_0$ is cyclic by Lemma \ref{fixedpoints}.

Let $G_0$ be generated by an element $g$ of order $d$. The set of fixed points of $G_0$ consists of isolated fixed points $E_i \cap E'_i$ and the curves $C$ and $C'$. Applying Lemma \ref{blowup} one can easily check that in the tangent space of $X$ at $E_i \cap E'_i$ the element $g$ acts as $\operatorname{diag}\left( \xi_{d}, \xi_{d}^{d - 1} \right)$. Thus there are four singularities of type $A_{d - 1}$ on $X / G_0$.

Let $f: X \rightarrow X / G_0$ be the quotient morphism and
$$
\pi: \widetilde{X / G_0} \rightarrow X / G_0
$$
\noindent be the minimal resolution of singularities. The transforms $\pi^{-1}f(E_i \cup E'_i)$ are chains of $d + 1$ curves with negative self-intersection numbers whose ends are $(-1)$-curves and the other curves are $(-2)$-curves (see Table \ref{table1}). One can $G / G_0$-equivariantly contract the curves that are the ends of these chains. By repeating this procedure we obtain a conic bundle $Y \rightarrow B$ without singular fibres if $\ord g$ is even, or with four singular fibres if $\ord g$ is odd. In the former case one has $K_Y^2 = 8$. In the latter case one can check that the self-intersection numbers of the proper transforms of $C$ and $C'$ on $Y$ are $-2$. Moreover, the components of any singular fibre of $Y \rightarrow B$ are contained in one $\left(G / G_0\right) \times\Gal\left( \kka / \ka\right)$-orbit, since $X$ is $G$-minimal and the components of any singular fibre of $X \rightarrow B$ are contained in one~$G\times\Gal\left( \kka / \ka\right)$-orbit. Therefore $Y$ is a $G / G_0$-minimal Iskovskikh surface.

\end{proof}

From now on assume that $G_0$ is trivial.

\begin{lemma}
\label{IskGB}
The group $G_B$ is either trivial or isomorphic to $\CG_2$. If $G_B$ is not trivial then the quotient $X / G_B$ is smooth, and $K_{X / G_B}^2 = 8$.
\end{lemma}

\begin{proof}

Let $g$ be any nontrivial element of $G_B$. If the curve $C$ is $g$-invariant then $gE_i = E_i$, $gE'_j = E'_j$ and $gC' = C'$. This contradicts the assumption that $G_0$ is trivial. Thus $gC = C'$, $gE_i = E'_i$, $gE'_j = E_j$ and $gC' = C$.

If two elements $g$ and $h$ act in such way, then the element $gh$ preserves all negative curves on $\XX$. Thus $gh$ is trivial and $G_B \cong \CG_2$.

The group $G_B$ has no isolated fixed points on $X$. Thus $X / G_B$ is smooth. Let~\mbox{$f: X \rightarrow X / G_B$} be the quotient map. The image of any smooth fibre of $X \rightarrow B$ is a smooth fibre of $X / G_B \rightarrow B$. For any~$i$ one has $gE_i = E'_i$ so the image of any singular fibre of $X \rightarrow B$ is a smooth fibre of $X / G_B \rightarrow B$ too. Therefore $X / G_B \rightarrow B$ is a conic bundle without singular fibres, and $K_{X / G_B}^2 = 8$.

\end{proof}

From now on assume that $G_B$ is trivial.

\begin{lemma}
\label{IskGFodd}
If the group $G$ is not isomorphic to $\CG_2$ or $\CG_2^2$ then the quotient $X / G$ is birationally equivalent to a surface $Y$ such that $K_Y^2 \geqslant 5$. 
\end{lemma}

\begin{proof}
This lemma directly follows from Theorem~\ref{Cbundle}.
\end{proof}

\begin{lemma}
\label{IskGFeven}
Assume that $G$ is isomorphic to $\CG_2$ or $\CG_2^2$, each nontrivial element $g \in G$ has only isolated fixed points on~$\XX$, and each pair of $g$-fixed points, lying in one fibre of $\overline{\varphi}:\XX \rightarrow \overline{B}$, is contained in one $G \times \Gal \left( \kka / \ka \right)$-orbit. Then the quotient $X / G$ is birationally equivalent to a surface~$Y$, admitting a structure of a minimal conic bundle $Y \rightarrow B$ such that $K_Y^2 = 4$. If one of the above assumptions fails then the quotient $X / G$ is birationally equivalent to a surface~$Y$ such that $K_Y^2 \geqslant 5$. 
\end{lemma}

\begin{proof}

Let $f: X \rightarrow X / G$ and $f': B \rightarrow B / G$ be the quotient morphisms,
$$
\pi: \widetilde{X / G} \rightarrow X / G
$$
\noindent be the minimal resolution of singularities, and $\psi: Y \rightarrow B / G$ be a relatively minimal model of $ \widetilde{X / G}$ over $B / G$.

Let $g$ be a nontrivial element in $G$. The element $g$ has two fixed points $q_g$ and $r_g$ on~$\overline{B}$, since $g$ faithfully acts on $\overline{B}$. By Lemma \ref{eveninvariantfibre} the points $q_g$ and $r_g$ differ from each of the points $p_1$, $p_2$, $p_3$, $p_4$. Moreover, if there is a nontrivial element $h \neq g$ in $G$, then $hq_g = r_g$ by Lemma \ref{fixedpoints}.

If $t$ is a point on $\overline{B}$ that differs from any $p_i$ and is not fixed by any nontrivial element of $G$, then the fibre of $\XX / G \rightarrow \overline{B} / G$ over $f'(t)$ is smooth. Therefore, for $G = \langle g \rangle \cong \CG_2$ the conic bundle $\overline{\psi}:\overline{Y} \rightarrow \overline{B} / G$ can have singular fibres only over points $f'(p_i) = f'(gp_i)$, $f'(q_g)$ and $f'(r_g)$. For $G = \langle g, h \rangle \cong \CG^2_2$ the conic bundle $\overline{\psi}$ can have singular fibres only over points
$$
f'(p_1) = f'(p_2) = f'(p_3) = f'(p_4), \quad f'(q_g) = f'(r_g), \quad f'(q_h) = f'(r_h), \quad f'(q_{gh}) = f'(r_{gh}).
$$
\noindent In both cases the conic bundle $\overline{\psi}$ has no more than $4$ singular fibres. Thus $K_Y^2 = 4$ if and only if the fibres over all mentioned points are singular. Otherwise, one has $K_Y^2 \geqslant 5$.

For any $i$ the curves $E_i$ and $E'_i$ are permuted by $G \times \Gal \left( \kka / \ka \right)$ since $X$ is $G$-minimal. Therefore the fibre of $\overline{\psi}$ over a point $f'(p_i)$ is singular.

Let $g$ be a nontrivial element of $G$, and $F_g$ be the fibre of $\XX \rightarrow \overline{B}$ over $q_g$. If $g$ acts on $F_g$ trivially then there are no singular points on $f(F_g)$, and the fibre of $\overline{\psi}$ over $f'(q_g)$ is smooth. If $g$ acts on $F_g$ faithfully then there are two fixed points of $g$ on $F_g$, and two $A_1$-singularities on $f(F_g)$. Thus the curve $\pi^{-1}f(F_g)$ is a chain of three $\kka$-rational curves with self-intersection numbers $-2$, $-1$ and $-2$ (see Table \ref{table1}). If the ends of this chain are permuted by $\Gal \left( \kka / \ka \right)$ then the fibre of $\overline{\psi}$ over $f'(q_g)$ is singular, and otherwise this fibre is smooth. The ends of this chain are permuted by $\Gal \left( \kka / \ka \right)$, if and only if the fixed points of $g$ on $F_g$ are permuted by $G \times \Gal \left( \kka / \ka \right)$. Therefore $K_Y^2 = 4$ if and only if each nontrivial element $g \in G$ has only isolated fixed points on $\XX$ and each pair of these points, lying in one fibre of $\overline{\varphi}$, is permuted by $G \times \Gal \left( \kka / \ka \right)$. Otherwise, one has $K_Y^2 \geqslant 5$.

\end{proof}

Now we can summarize the results of Lemmas \ref{IskG0}, \ref{IskGB}, \ref{IskGFodd} and \ref{IskGFeven} in the following corollary.

\begin{corollary}
\label{IskGfixed}
Let a group $G$ of order $2^n$ act on an Iskovskikh surface $X$. Suppose that there is an element $g$ in $G$ such that $g$ has a curve of fixed points. Then the quotient $X / G$ is birationally equivalent to a surface $Y$ such that $K_Y^2 \geqslant 5$. In particular, if $X$ contains a $\ka$-point that does not lie on a section with self-intersection number $-2$ then $X / G$ is $\ka$-rational. 
\end{corollary}

\begin{proof}
If we find a normal group $N$ in $G$ such that the quotient $X / N$ is \mbox{$G / N$-birationally} equivalent to a surface $Z$ with $K_Z^2 \geqslant 5$, then we are done by Remark \ref{geq5touse} \mbox{since $X /G \approx Z / (G / N)$}. Assume that there is no such subgroup in $G$.

Then by Lemma \ref{IskG0} the group $G_0$ is trivial, and by Lemma \ref{IskGB} the group $G_B$ is trivial. Therefore by Lemma \ref{IskGFodd} one has $G \cong \CG_2$ or $G \cong \CG_2^2$. By Lemma \ref{IskGFeven} for these cases the quotient $X / G$ is birationally equivalent to a surface $Y$ such that $K_Y^2 \geqslant 5$.

By Lemma \ref{IskUnirat} if $X$ contains a $\ka$-point that does not lie on a section with self-intersection number $-2$ then $X(\ka)$ is dense. Therefore $Y(\ka)$ is dense and $X / G \approx Y$ is $\ka$-rational by Theorem \ref{ratcrit}.

\end{proof}

\section{Del Pezzo surfaces of degree $2$}

A del Pezzo surface $X$ of degree $2$ is isomorphic to a smooth quartic surface in a \mbox{$\ka$-form} of the weighted projective space $\Pro_{\ka}(1 : 1 : 1 : 2)$. The anticanonical map gives a double cover of $\Pro^2_{\ka}$ branched in a smooth quartic curve. The corresponding involution $\gamma$ on $X$ is called \textit{the Geiser involution}. Obviously the involution $\gamma$ commutes with any element of $\Aut(X)$. In particular, $\langle \gamma \rangle$ is a normal subgroup in $\Aut(X)$. Thus if a finite group $G \subset \Aut(X)$ contains $\gamma$ then the quotient $X / G \approx \Pro^2_{\ka} / (G / \langle \gamma \rangle)$ is birationally equivalent to a surface $Y$ such that $K_Y^2 \geqslant 5$ by Remark \ref{geq5touse}.

The following theorem classifies actions of cyclic groups of prime order on del Pezzo surfaces of degree $2$.

\begin{theorem}[{cf. \cite[Lemma~6.16]{DI1}}]
\label{DP2cyclicclass}
Let a cyclic group of prime order $\CG_p$ act on a del Pezzo surface of degree $2$. Then one can choose coordinates in $\Pro_{\kka}(1 : 1: 1: 2)$ such that the equation of $\XX$ and the action of $\CG_p$ are presented in Table \ref{table2}.

\vbox{
\begin{center}
\begin{longtable}{|c|c|c|c|}
\caption[]{\label{table2} automorphisms of prime order}\endhead
\hline
Type & Order & Equation & Action \\
\hline
$0$ & $2$ & $L_4(x,y,z) + t^2 = 0$ & $\left( x : y : z : -t \right)$ \\
\hline
$1$ & $2$ & $L_4(x,y) + L_2(x,y)z^2 + z^4 + t^2 = 0$ & $\left( x : y : -z : t \right)$ \\
\hline
$2$ & $2$ & $L_4(x,y) + L_2(x,y)z^2 + z^4 + t^2 = 0$ & $\left( x : y : -z : -t \right)$ \\
\hline
$3$ & $3$ & $L_4(x,y) + L_1(x,y)z^3 + t^2 = 0$ & $\left( x : y : \omega z : t \right)$ \\
\hline
$4$ & $3$ & $(x^3 + y^3)z + Ax^2y^2 + 2Bxyz^2 + z^4 + t^2 = 0$ & $\left( \omega x : \omega^2 y : z : t \right)$ \\
\hline
$5$ & $7$ & $x^3y + y^3z + z^3x + t^2 = 0$ & $\left( \xi_7 x : \xi_7^4 y : \xi_7^2 z : t \right)$ \\
\hline
\end{longtable}

\end{center}}

where $L_k$ is a homogeneous polynomial of degree $k$, and $A$ and $B$ are elements of $\kka$.

\end{theorem}

In what follows we refer to the conjugacy classes of elements of finite order acting on a del Pezzo surface of degree $2$ as elements of type $1$, type $2$, etc.

The element of type $0$ is the Geiser involution.

In this section we prove the following proposition.

\begin{proposition}
\label{DP2}
Let $X$ be a del Pezzo surface of degree $2$ and $G$ be a finite subgroup of~$\Aut_{\ka}(X)$ such that there is a smooth $\ka$-point on $X / G$. Then $X / G$ can be non-$\ka$-rational only if the group $G$ is one of the following:
\begin{enumerate}
\item $G$ is trivial;
\item $G \cong \CG_2$ is generated by an element of type $1$ (see Table \ref{table2});
\item $G \cong \CG_2$ is generated by an element of type $2$ (see Table \ref{table2});
\item $G \cong \CG_3$ is generated by an element of type $4$ (see Table \ref{table2});
\item $G \cong \CG_4$ is generated by an element whose action is conjugate to $\left( \ii x : -\ii y : z : t \right)$;
\item $G \cong \CG_4$ is generated by an element whose action is conjugate to $\left( \ii x : -\ii y : z : -t \right)$;
\item $G \cong \CG_2^2$ is generated by elements of type $2$ (see Table \ref{table2});
\item $G \cong \SG_3$ is generated by elements of type $2$ (see Table \ref{table2});
\item $G \cong \DG_8$ is generated by elements of type $2$ (see Table \ref{table2});
\item $G \cong Q_8$ is generated by elements whose actions are conjugate to $\left( \ii x : -\ii y : z : t \right)$;
\item $G \cong Q_8$ is generated by elements whose actions are conjugate to $\left( \ii x : -\ii y : z : -t \right)$.
\end{enumerate}
For any other group $G$ the quotient $X / G$ is $\ka$-rational.
\end{proposition}

To prove Proposition \ref{DP2} we need to list all possible groups of automorphisms of del Pezzo surfaces of degree $2$.

\begin{theorem}[{cf. \cite[Subsection 6.6, Table 6]{DI1}}]
\label{DP2groupclass}
Let $\XX$ be a del Pezzo surface of degree~$2$. Then for each possibility of $\Aut(\XX)$ one can choose coordinates in $\Pro_{\kka}(1 : 1 : 1 : 2)$ such that the equation of $\XX$ and the group $\Aut(\XX)$ are presented  in Table \ref{table3}.

\vbox{
\begin{center}
\begin{longtable}{|c|c|c|}
\caption[]{\label{table3} automorphisms groups}\endhead
\hline
Type & Group & Equation  \\
\hline
$I$ & $\CG_2 \times \mathrm{PSL}_2 \left( \F_7 \right)$ & $x^3y + y^3z + z^3x + t^2 = 0$ \\
\hline
$II$ & $\CG_2 \times \left( \CG_4^2 \rtimes \SG_3 \right)$ & $x^4 + y^4 + z^4 + t^2 = 0$ \\
\hline
$III$ & $\CG_2 \times \widetilde{\AG}_4 $ & $x^4 + Ax^2y^2 + y^4 + z^4 + t^2 = 0$ \\
\hline
$IV$ & $\CG_2 \times \SG_4 $ & $x^4 + y^4 + z^4 + A(x^2y^2 + x^2z^2 + y^2z^2) + t^2 = 0$ \\
\hline
$V$ & $\CG_2 \times \mathrm{AS}_{16} $ & $x^4 + Ax^2y^2 + y^4 + z^4 + t^2 = 0$ \\
\hline
$VI$ & $\CG_2 \times \CG_9 $ & $x^4 + xy^3 + yz^3 + t^2 = 0$ \\
\hline
$VII$ & $\CG_2 \times \DG_8 $ & $x^4 + Ax^2y^2 + y^4 +Bxyz^2 + z^4 + t^2 = 0$ \\
\hline
$VIII$ & $\CG_2 \times \CG_6 $ & $x^4 + Ax^2y^2 + y^4 + xz^3 + t^2 = 0$ \\
\hline
$IX$ & $\CG_2 \times \SG_3 $ & $(x^3 + y^3)z + Ax^2y^2 + Bxyz^2 + z^4 + t^2 = 0$ \\
\hline
$X$ & $\CG_2 \times \CG_2^2 $ & $x^4 + y^4 + z^4 + Ax^2y^2 + Bx^2z^2 + Cy^2z^2 + t^2 = 0$ \\
\hline
$XI$ & $\CG_2 \times \CG_3 $ & $xz^3 + L_4(x,y) + t^2 = 0$ \\
\hline
$XII$ & $\CG_2 \times \CG_2 $ & $z^4 + z^2L_2(x,y) + L_4(x,y) + t^2 = 0$ \\
\hline
$XIII$ & $\CG_2 $ & $L_4(x,y,z) + t^2 = 0$ \\
\hline

\end{longtable}
\end{center}}

In the second column the first factor $\CG_2$ is the group generated by the Geiser involution. The group $\widetilde{\AG}_4$ is a central extention of $\AG_4$, and $\left| \widetilde{\AG}_4 \right| = 48$. The group $\mathrm{AS}_{16}$ is a group of order $16$ generated by elements $\operatorname{diag}(\ii,\ii,1,1)$, $\operatorname{diag}(\ii,-\ii,1,1)$ and a permutation of $x$ and~$y$.

In the third column $L_k$ is a homogeneous polynomial of degree $k$, and $A$, $B$ and $C$ are elements of $\kka$.

\end{theorem}

In the paper \cite{DI1} there is one more column in this table which contains conditions on the parameters. But we are interested only in the structure of the group and its action on~$\Pro_{\kka}(1 : 1 : 1 : 2)$ so we omit this column.

For shortness, we will say that a surface $X$ has type $I$, type $II$, etc. of Theorem \ref{DP2groupclass}, if the surface $\XX$ has the corresponding type.

\begin{remark}
\label{DP2typeIspec}
Note that one can find the coordinates in which the surface of type $I$ is given by the equation of a surface of type~$IV$ for $A = \frac{\sqrt{-7} - 1}{2}$ (see \cite[Subsection 6.6]{DI1}).
\end{remark}

\subsection{$2$-groups}

Let $X$ be a del Pezzo surface of degree $2$, and $G \subset \Aut_{\ka}(X)$ be a $2$-group. In this subsection we study, for which groups the quotient $X / G$ can be non-$\ka$-rational. We assume that $G$ does not contain the Geiser involution.

\begin{lemma}
\label{DP2type1}
Let a finite group $G$ act on a del Pezzo surface $X$ of degree $2$ and $N \cong \CG_2$ be a normal subgroup in $G$ generated by an element of type $1$. Then the quotient $X / N$ is $G / N$-birationally equivalent to an Iskovskikh surface.
\end{lemma}

\begin{proof}
By Theorem \ref{DP2cyclicclass} one can choose coordinates in $\Pro_{\kka}(1 : 1 : 1 : 2)$ such that the set of $N$-fixed points on $\XX$ consists of a fixed curve $z = 0$, that has class $-K_X$, and two isolated fixed \mbox{points $(0 : 0 : 1 :\pm \ii)$}. Let~\mbox{$f: X \rightarrow X / N$} be the quotient morphism. By the Hurwitz formula (see \cite[Equation (1.38)]{Dol12}) one has $K_X = f^*(K_{X / N}) - K_X$. Therefore
$$
K_{X / N}^2 = \frac{1}{2}(2K_X)^2 = 4
$$
and the surface $X / N$ is a singular del Pezzo surface with two $A_1$-singularities. Let $\pi: \widetilde{X / N} \rightarrow X / N$ be the minimal resolution of singularities.

Let $\mathcal{L}$ be a linear system on $X$ given by $\lambda x = \mu y$. A general member of $\mathcal{L}$ is an \mbox{$N$-invariant} elliptic curve passing through the isolated fixed points of $N$. Therefore a general member of the linear system $f_*(\mathcal{L})$ is a conic passing through the two singular points on~$X / N$. Hence the linear system $\pi^{-1}_*f_*(\mathcal{L})$ defines a structure of a conic bundle on~$\widetilde{X / N}$, the exceptional curves over $A_1$-singularities are sections of this conic bundle with self-intersection number~$-2$, and $K_{\widetilde{X / N}}^2 = 4$. Thus $\widetilde{X / N}$ is an Iskovskikh surface.
\end{proof}

\begin{remark}
\label{DP2type1min}
Note that the obtained Iskovskikh surface $\widetilde{X / N}$ is $G / N$-minimal, if and only if $\rho(X)^G = 1$ and the two isolated fixed points of $N$ are permuted by the \mbox{group $G \times \Gal \left( \kka / \ka \right)$}, since otherwise $\rho\left( \widetilde{X / N} \right)^{G / N} > 2$.
\end{remark}

\begin{corollary}
\label{DP2type1fixed}
Let a finite group $G$ of order $2^n$ act on a del Pezzo surface $X$ of degree~$2$ and $N \cong \CG_2$ be a normal subgroup generated by an element of type $1$. Assume that there is an element $g \notin N$ in $G$ such that $g$ has a curve of fixed points on $X$, or $g$ acts on the tangent spaces of $X$ at the isolated fixed points by multiplication by a scalar matrix. Then the quotient $X / G$ is birationally equivalent to a surface $Y$ such that $K_Y^2 \geqslant 5$.
\end{corollary}

\begin{proof}
By Lemma \ref{DP2type1} the quotient $X / N$ is $G / N$-birationally equivalent to Iskovskikh surface $Z$. It follows from the assumption that in the group $G / N$ acting on $Z$ there is an element, that has a curve of fixed points on $Z$. Thus the quotient \mbox{$X / G \approx Z / (G / N)$} is birationally equivalent to a surface $Y$ such that $K_Y^2 \geqslant 5$ by Corollary~\ref{IskGfixed}.

\end{proof}

\begin{lemma}
\label{DP2type2}
Let a finite group $G$ act on a del Pezzo surface $X$ of degree $2$ and $N \cong \CG_2$ be a normal subgroup in $G$ generated by an element of type $2$. Then the quotient $X / N$ is $G / N$-birationally equivalent to a del Pezzo surface $Y$ of degree $2$.
\end{lemma}

\begin{proof}

By Theorem \ref{DP2cyclicclass} one can choose coordinates in $\Pro_{\kka}(1 : 1 : 1 : 2)$ such that the set of $N$-fixed points on $\XX$ consists of four isolated fixed points $r_i = (\lambda_i : \mu_i : 0 : 0)$ where~$(\lambda_i : \mu_i)$ are the roots of the equation $L_4(x,y) = 0$. The curve $C$ given by $z = 0$ is $N$-invariant and passes through all $N$-fixed points.

Let $f: X \rightarrow X / N$ be the quotient morphism and
$$
\pi: \widetilde{X / N} \rightarrow X / N
$$
\noindent be the minimal resolution of singularities. Then
$$
\pi^{-1}_*f(C)^2 = f(C)^2 - 4 \cdot \frac{1}{2} = \frac{1}{2} C^2 - 2 = -1.
$$
Let $\widetilde{X / N} \rightarrow Y$ be the contraction of the $(-1)$-curve $\pi^{-1}_*f(C)$. The surface $X / N$ is a singular del Pezzo surface. Therefore $\widetilde{X / N}$ is a weak del Pezzo surface. There are no negative curves with self-intersection less than $-1$ on $Y$, thus the surface $Y$ is a del Pezzo surface by Proposition \ref{DPconnection}. Its degree is equal to
$$
K_Y^2 = K_{\widetilde{X / N}}^2 + 1 = K_{X / N}^2 + 1 = \frac{1}{2} K_X^2 + 1 = 2.
$$
\end{proof}

\begin{remark}
\label{DP2type2min}
Assume that $G = N$, $\rho(X)^G = 1$ and let us show that under the assumptions of Lemma \ref{DP2type2} the surface $Y$ is not $\ka$-rational. Let $S_1$, $S_2$, $S_3$ and $S_4$ be the proper transforms of the curves $\pi^{-1}_*f(r_i)$ on $Y$. The surface $Y$ is a del Pezzo surface of degree $2$, therefore the anticanonical map $Y \rightarrow \Pro^2_{\ka}$ defines the action of the Geiser involution on~$Y$. Denote by $T_i$ the image of $S_i$ under this action. One has
$$
S_i^2 = T_i^2 = -1, \qquad S_i \cdot S_j = T_i \cdot T_j = 1, \qquad S_i \cdot T_i = 2, \qquad S_i \cdot T_j = 0.
$$

Assume that the Galois group of the polynomial $L_4(x, y)$ is trivial. Then the curves $S_i$ and $T_i$ are defined over $\ka$, and $\rho(Y) = 4$. The complete linear system $|S_1 + S_2|$ gives a conic bundle structure $\varphi: Y \rightarrow \Pro^1_{\ka}$, and the class of $T_3 + T_4$ belongs to this linear system. One can contract the curves $S_1$ and $T_3$, which are the components of singular fibres of~$\varphi$, and get a conic bundle $Z \rightarrow \Pro^1_{\ka}$ such that $K_Z^2 = 4$ and $\rho(Z) = 2$. The surface $Z$ is minimal by Theorem \ref{MinCB}, and is not $\ka$-rational by Theorem \ref{ratcrit}.

Obviously, if the polynomial $L_4(x,y)$ has nontrivial Galois group then $Y$ is also \mbox{non-$\ka$-rational}. Minimal model $Z$ of $Y$ depends on the Galois group of $L_4(x,y)$. One can check that:

\begin{itemize}

\item if the Galois group transitively permutes the roots of $L_4(x,y) = 0$, then $\rho(Y) = 1$;

\item if any root of $L_4(x,y) = 0$ is not defined over $\ka$, and the roots of $L_4(x,y) = 0$ form two $\Gal\left( \kka / \ka \right)$-orbits, then $\rho(Y) = 2$ and $Y$ admits a structure of a minimal conic bundle;

\item if there is a unique root of $L_4(x,y) = 0$ defined over $\ka$, then $Z$ is a minimal cubic surface;

\item if the equation $L_4(x,y) = 0$ has exactly two roots defined over $\ka$, then $Z$ is a minimal del Pezzo surface of degree $4$, and $\rho(Z) = 1$;

\item if all roots of the equation $L_4(x,y) = 0$ are defined over $\ka$, then $Z$ is a minimal del Pezzo surface of degree $4$, admitting a structure of a minimal conic bundle.

\end{itemize}

\end{remark}

\begin{lemma}
\label{DP2V4}
Let a finite group $G$ act on a del Pezzo surface $X$ of degree $2$ and \mbox{$N \cong \VG_4$} be a normal subgroup in $G$ generated by elements of type $1$. Then the quotient $X / G$ is birationally equivalent to a surface $Y$ such that $K_Y^2 \geqslant 5$.
\end{lemma}

\begin{proof}
Each nontrivial element of $N$ has a pointwisely fixed hyperplane section and two isolated fixed points lying on the fixed curves of other nontrivial elements. Thus the quotient $X / N$ is smooth and by the Hurwitz formula
$$
K_{X / N}^2 = \frac{1}{4}(4K_X)^2 = 8.
$$
By Remark \ref{geq5touse} the quotient $X / G \approx (X / N) / (G / N)$ is birationally equivalent to a surface~$Y$ such that $K_Y^2 \geqslant 5$.
\end{proof}

Until the end of this paper we will use the following notation.

\begin{notation}
\label{DP2notation}

Let $\alpha$, $\beta$, $\delta$ and $\gamma$ be the following automorphisms of $\Pro_{\kka} (1 : 1 : 1 : 2)$:
$$
\alpha: (x : y: z : t) \mapsto (\ii x : y : z : t); \qquad \beta: (x : y : z : t) \mapsto (x : \ii y : z : t);
$$
$$
\delta: (x : y: z : t) \mapsto (y : x : z : t); \qquad \gamma: (x : y : z : t) \mapsto (x : y : z : -t).
$$
One has $\langle \alpha, \beta, \delta, \gamma \rangle \cong \left((\CG_4 \times \CG_4) \rtimes \CG_2 \right) \times \CG_2$.

\end{notation}

\begin{lemma}
\label{DP22groups}
If $\XX$ is a del Pezzo surface of degree $2$ then one can choose the coordinates in $\Pro_{\kka}(1 : 1 : 1 : 2)$ in which there is a $2$-Sylow subgroup $S$ in $\Aut(\XX)$ such that $S \subset \langle \alpha, \beta, \delta, \gamma \rangle$. 
\end{lemma}

\begin{proof}
One can easily find such $2$-Sylow subgroup for each type (see Theorem \ref{DP2groupclass}) of $\XX$:
\begin{itemize}
\item $S = \langle \alpha, \beta, \delta, \gamma \rangle$ for the case $II$;
\item $S = \langle \alpha \beta, \alpha \beta^{-1}, \delta, \gamma \rangle$ for the cases $III$ and $V$;
\item $S = \langle \alpha^2, \beta^2, \delta, \gamma \rangle$ for the cases $I$ (see Remark \ref{DP2typeIspec}), $IV$ and $VII$;
\item $S = \langle \alpha^2, \beta^2, \gamma \rangle$ for the case $X$;
\item $S = \langle \beta^2, \gamma \rangle$ for the case $VIII$;
\item $S = \langle \delta, \gamma \rangle$ for the case $IX$;
\item $S = \langle \alpha^2 \beta^2, \gamma \rangle$ for the case $XII$;
\item $S = \langle\gamma \rangle$ for the cases $VI$, $XI$ and $XIII$.
\end{itemize}
\end{proof}

\begin{lemma}
\label{DP2even}
Let a finite group $G$ of order $2^n$ act on a del Pezzo surface $X$ of degree~$2$. If the group $G$ is not listed in Proposition \ref{DP2} then the quotient~$X / G$ is birationally equivalent to a surface $Y$ such that $K_Y^2 \geqslant 5$.
\end{lemma}

\begin{proof}
Assume that the quotient $X / G$ is not birationally equivalent to a surface $Y$ such that $K_Y^2 \geqslant 5$.

By Lemma \ref{DP22groups} we can assume that $G$ is a subgroup of $\langle \alpha, \beta, \delta, \gamma \rangle$.

The group $G$ does not contain the Geiser involution $\gamma$ since the quotient \mbox{$X / G \approx \Pro^2_{\ka} / (G / \langle \gamma \rangle)$} is birationally equivalent to a surface $Y$ such that $K_Y^2 \geqslant 5$ by Remark~\ref{geq5touse}.

Let $G_0 = G \cap \langle \alpha, \beta \rangle$ be a subgroup. Obviously $G_0$ is normal in $G$. If $G_0$ is trivial then $G$ is $\CG_2$ generated by an element of type $1$ or $2$, or $\CG_2^2$ generated by elements of type $2$. These groups are listed in cases $(2)$, $(3)$ and $(7)$ of Proposition \ref{DP2}.

If $G_0$ is not trivial then there is an element of order $2$ in $G_0$. If $G$ is a subgroup of~$\langle \alpha, \beta, \gamma \rangle$ then without loss of generality we can assume that $G_0$ contains $\alpha^2\beta^2$, since we can rename coordinates $x$, $y$ and $z$ in this case. If $G$ is not a subgroup of $\langle \alpha, \beta, \gamma \rangle$ then there is an element $h \in G$ such that $h\alpha^2h^{-1} = \beta^2$. Therefore in this case $G_0$ always contains~$\alpha^2\beta^2$. The subgroup $N$ generated by $\alpha^2\beta^2$ is normal in $G$.

By Corollary \ref{DP2type1fixed} the group $G_0$ does not contain elements $\alpha$, $\alpha^2$, $\alpha^3$, $\beta$, $\beta^2$, $\beta^3$, $\alpha \beta$ and~$\alpha^3\beta^3$. Therefore we have $G_0 = \langle \alpha^2 \beta^2 \rangle$ or $G_0 = \langle \alpha^3\beta \rangle$.

Now consider a group $G_1 = G \cap \langle \alpha, \beta, \gamma \rangle$. For an element $g = \alpha^i\beta^j\gamma$ one \mbox{has $g^2 = \alpha^{2i}\beta^{2j}$}. Therefore $i + j$ is even. Moreover, by Corollary \ref{DP2type1fixed} the group $G_1$ does not contain elements~$\alpha \beta\gamma$ and $\alpha^3\beta^3\gamma$. One has $G_1 = \langle \alpha^2\beta^2 \rangle$, $G_1 = \langle \alpha^3\beta \rangle$, $G_1 = \langle \alpha^3\beta\gamma \rangle$ \mbox{or $G_1 = \langle \alpha^2\gamma, \beta^2\gamma \rangle$}. Note that these groups are listed in cases $(2)$, $(5)$, $(6)$ and $(7)$ of Proposition \ref{DP2}.

Now assume that $G \neq G_1$. For an element $g = \alpha^i\beta^j\delta\gamma^k$ one \mbox{has $g^2 = \alpha^{i+j}\beta^{i+j}$}. Therefore $i + j$ is even, and if $\ord g = 4$ then $i + j$ is not divisible by $4$. One can easily check that any element of order $4$ in $G$ is conjugate in $\Aut \Pro_{\kka}(1:1:1:2)$ to $\operatorname{diag}(\ii, -\ii, 1, \pm 1)$. Therefore if $G_1 = \langle \alpha^2\beta^2 \rangle$ then either $G \cong \CG_4$, or $G \cong \CG_2^2$ generated by elements of type $2$. These groups are listed in cases $(5)$, $(6)$ and $(7)$ of Proposition \ref{DP2}.

By Corollary \ref{DP2type1fixed} the group $G_1$ does not contain elements $\delta$, $\alpha^3\beta \delta$, $\alpha^2\beta^2\delta$ or $\alpha \beta^3\delta$. \mbox{If $G_1 = \langle \alpha^3\beta \rangle$} then $G = \langle \alpha^3\beta, \alpha \beta \delta \rangle$, $G = \langle \alpha^3\beta, \alpha \beta \delta\gamma \rangle$ or $G = \langle \alpha^3\beta, \delta\gamma \rangle$. These groups are listed in cases $(10)$, $(11)$ and $(9)$ of Proposition \ref{DP2}.

If $G_1 = \langle \alpha^3\beta\gamma \rangle$ then $G = \langle \alpha^3\beta\gamma, \alpha^2\delta \rangle$ or $G = \langle \alpha^3\beta\gamma, \alpha \beta \delta \rangle$. These groups are listed in case~$(11)$ of Proposition \ref{DP2}.

If $G_1 = \langle \alpha^2\gamma, \beta^2\gamma \rangle$ then $G = \langle \alpha^2\gamma, \beta^2\gamma, \alpha^2\delta \rangle$ or $G = \langle \alpha^2\gamma, \beta^2\gamma, \alpha \beta \delta \rangle$. These groups are listed in case $(9)$ of Proposition \ref{DP2}.

\end{proof}

To construct explicit examples of $\ka$-rational an non-$\ka$-rational quotients we need the following corollary.

\begin{corollary}
\label{DP2evencrit}
Assume that a group $G$ of order $2^n$, listed in Proposition \ref{DP2}, acts on a del Pezzo surface $X$ of degree $2$ and contains a normal subgroup $N \cong \CG_2$ generated by an element of type $1$, and $\rho(X)^G = 1$. Let $\mathcal{L}$ be a linear subsystem of $|-K_{\XX}|$ that consists of all curves passing through the two isolated fixed points of $N$.

Assume that the pair of isolated fixed points of $N$ is permuted by the group $G \times \Gal \left( \kka / \ka \right)$, and for each element $g \in G$, $g \notin N$, fixed points of $g$, lying on a $g$-invariant member $R$ of~$\mathcal{L}$, are contained in one $G \times \Gal \left( \kka / \ka \right)$-orbit. Then the quotient $X / G$ is birationally equivalent to a surface~$Y$, admitting a structure of a minimal conic bundle~$Y \rightarrow B$ such that $K_Y^2 = 4$. In all other cases the quotient $X / G$ is birationally equivalent to a surface~$Y$ such that $K_Y^2 \geqslant 5$. In particular, if $Y(\ka) \neq \varnothing$ then $X / G$ is $\ka$-rational.
\end{corollary}

\begin{proof}
If the two isolated fixed points of $N$ are not permuted by the group $G \times \Gal \left( \kka / \ka \right)$ then by Remarks \ref{DP2type1min} and \ref{geq5touse} the quotient $X / G$ is birationally equivalent to a surface~$Y$ such that~$K_Y^2 \geqslant 5$.

Otherwise by Remark \ref{DP2type1min} the quotient $X / N$ is $G / N$-birationally equivalent to a $G / N$-minimal Iskovskikh surface $Z$, and members of $\mathcal{L}$ correspond to fibres of conic \mbox{bundle $Z \rightarrow \Pro^1_{\ka}$}. The group $G / N$ is either trivial, or isomorphic to $\CG_2$ or $\CG_2^2$, and its nontrivial elements have only isolated fixed points on $Z$. Thus for the quotient $Z / (G / N) \approx X / G$ the assertion immediately follows from Lemma \ref{IskGFeven}.

\end{proof}

\subsection{Groups of order divisible by $3$}

Let $X$ be a del Pezzo surface of degree $2$, and $G \subset \Aut_{\ka}(X)$ be a group of order $2^k \cdot 3^n$. In this subsection we study, for which groups the quotient $X / G$ can be non-$\ka$-rational. We assume that $G$ does not contain the Geiser involution.

\begin{lemma}
\label{DP2type3}
Let a finite group $G$ act on a del Pezzo surface $X$ of degree $2$ and $N \cong \CG_3$ be a normal subgroup in $G$ generated by an element of type $3$. Then the quotient $X / G$ is birationally equivalent to a surface $Y$ such that $K_Y^2 \geqslant 5$.
\end{lemma}

\begin{proof}

The set of $N$-fixed points consists of a fixed curve $z = 0$ and a fixed \mbox{point $(0 : 0 : 1 : 0)$}. By the Hurwitz formula
$$
K_{X / N}^2 = \frac{1}{3}(3K_X)^2 = 6,
$$
\noindent and the surface $X / N$ is a singular del Pezzo surface with one $A_2$-singularity. Let~\mbox{$\widetilde{X / N} \rightarrow X / N$} be the minimal resolution of singularities. Then $K_{\widetilde{X / N}}^2 = 6$. By Remark~\ref{geq5touse} the quotient $X / G \approx \left( \widetilde{X / N} \right) / (G / N)$ is birationally equivalent to a surface~$Y$ such that $K_Y^2 \geqslant 5$.

\end{proof}

\begin{lemma}
\label{DP2type4}
Let a finite group $G$ act on a del Pezzo surface $X$ of degree $2$ and $N \cong \CG_3$ be a normal subgroup in $G$ generated by an element of type $4$. Then the quotient $X / N$ is $G / N$-birationally equivalent to a surface $Y$ such that $K_Y^2 = 4$.
\end{lemma}

\begin{proof}

The group $N$ has four isolated fixed points $p_1 = (1: 0 : 0 : 0)$, $p_2 = (0: 1 : 0 :0)$ and $q_{1, 2} = (0 : 0 : 1 : \pm \ii)$. On the tangent spaces of $X$ at the points $p_i$ the group $N$ acts as $\langle\operatorname{diag}(\omega, \omega)\rangle$ and on the tangent spaces of $X$ at the points $q_i$ the group $N$ acts as~$\langle\operatorname{diag}(\omega, \omega^2)\rangle$.

Let $C_1$ and $C_2$ be curves given by $y = 0$ and $x = 0$ respectively. These curves are $N$-invariant. The curve $C_i$ passes through the points $q_1$, $q_2$ and $p_i$. The section $z = 0$ consists of two irreducible components $D_1$ and $D_2$ being $(-1)$-curves each passing through $p_1$ and $p_2$.

The surface $X / N$ has two $A_2$-singularities and two $\frac{1}{3}(1,1)$-singularities. Let \mbox{$f: X \rightarrow X / N$} be the quotient morphism and
$$
\pi: \widetilde{X / N} \rightarrow X / N
$$
\noindent be the minimal resolution of singularities. One can easily check that the proper transforms $\pi^{-1}_*f(C_i)$ and $\pi^{-1}_*f(D_i)$ are four disjoint $(-1)$-curves (see Table \ref{table1}). Let $h: \widetilde{X / N} \rightarrow Y$ be the \mbox{$G / N$-equivariant} contraction of these curves. Then
$$
K_Y^2 = K_{\widetilde{X / N}}^2 + 4 = K_{X/ N}^2 - \frac{2}{3} + 4 = \frac{1}{3}K_X^2 + \frac{10}{3} = 4.
$$

\end{proof}

\begin{remark}
\label{DP2type4min}

Assume that $\rho\left( X \right)^G = 1$. In this case the $(-1)$-curves $D_1$ and $D_2$ are permuted by the group $G \times \Gal \left( \kka / \ka \right)$, since otherwise one can contract one of these curves. We want to find conditions, when the surface $Y$ is $G / N$-minimal.

The points $p_1$ and $p_2$ lie on the ramification divisor of the map $X \rightarrow \Pro^2_{\ka}$, and the points $q_1$ and $q_2$ do not. Therefore there are five possibilities of the action of the group $G \times \Gal \left( \kka / \ka \right)$ on the set of $N$-fixed points:

\begin{enumerate}
\item $G \times \Gal \left( \kka / \ka \right)$ fixes $p_1$, $p_2$, $q_1$ and $q_2$;
\item $G \times \Gal \left( \kka / \ka \right)$ fixes $p_1$ and $p_2$, and permutes $q_1$ and $q_2$;
\item $G \times \Gal \left( \kka / \ka \right)$ permutes $p_1$ and $p_2$, and fixes $q_1$ and $q_2$;
\item the points $p_1$ and $p_2$, and the points $q_1$ and $q_2$ are permuted by the same elements of $G \times \Gal \left( \kka / \ka \right)$;
\item the points $p_1$ and $p_2$, and the points $q_1$ and $q_2$ are permuted by different elements of $G \times \Gal \left( \kka / \ka \right)$.
\end{enumerate}

Note that a curve $h\pi^{-1}\left(f(q_i)\right)$ is reducible and consists of two \mbox{$(-1)$-curves} $R_{i1}$ and $R_{i2}$ meeting each other at a point. One has
$$
R_{11} \cdot R_{21} = R_{12} \cdot R_{22} = 1, \qquad R_{11} \cdot R_{22} = R_{12} \cdot R_{21} = 0.
$$

In cases $(1)$ and $(4)$ one can $G / N$-equivariantly contract two $(-1)$-curves $R_{11}$ and~$R_{22}$, and get a surface $Z$ such that $K_Z^2 = 6$. In both cases $(2)$ and $(3)$ one has $\rho(Y)^{G / N} = 2$, and the complete linear systems $|R_{11} + R_{21}|$ and $|R_{11} + R_{12}|$ respectively give a structure of a $G / N$-equivariant conic bundle on $Y$. Thus in these cases $Y$ is \mbox{$G /N$-minimal} by Theorem~\ref{MinCB}. In case $(5)$ one has $\rho(Y)^{G / N} = 1$, and $Y$ is a $G / N$-minimal del Pezzo surface.

\end{remark}

\begin{remark}
\label{DP2type4points}
Note that in cases $(1)$ and $(2)$ the set $Y(\ka)$ is dense, since there is a \mbox{$\ka$-point~$h\pi^{-1}_*f(C_1)$} on $Y$, and by \cite[Theorem IV.7.8]{Man67} the surface $Y$ is $\ka$-rational or $\ka$-unirational of degree $2$ (i.e. birationally equivalent to a quotient of a $\ka$-rational surface by an involution).
\end{remark}

\begin{lemma}
\label{DP2odd1}
Let a finite group $G$ of order $2^k \cdot 3^n$, $n \geqslant 1$, act on a del Pezzo surface $X$ of degree $2$ and $X$ be of type $III$, $VI$, $VIII$ or $XI$ (see Theorem \ref{DP2groupclass}). Then the quotient~$X / G$ is birationally equivalent to a surface $Y$ such that $K_Y^2 \geqslant 5$.
\end{lemma}

\begin{proof}

Note that if $X$ has type $VI$, $VIII$ or $XI$ then $G$ contains an element of type $3$ that generates a normal subgroup $N$. Therefore $X / G$ is birationally equivalent to a surface $Y$ such that $K_Y^2 \geqslant 5$ by Lemma \ref{DP2type3}.

Assume that $X$ is a surface of type $III$. The group $\Aut(\XX)$ is generated by elements $\alpha \beta$, $\alpha \beta^3$, $\delta$, $\gamma$ (see Notation \ref{DP2notation}) and an element $\zeta$ of order $3$ such that $\zeta \alpha \beta^3\zeta^{-1} = \alpha^2\delta$ and $\zeta \alpha \beta \zeta^{-1} = \alpha \beta$. All elements of order $3$ are conjugate in $\Aut(\XX)$, therefore we can assume that $G$ contains the element $\zeta$.

One can check that the element $\zeta$ acts in the following way
$$
\zeta(x : y : z: t) = \left( \frac{-(1+\ii)x-(1+\ii)y}{2} : \frac{(1-\ii)x - (1 - \ii)y}{2} : \omega z : \omega^2 t \right).
$$
So the element $\zeta$ is of type $3$ in the notation of Table \ref{table2}.

If the group $G$ contains the element $\alpha^2\beta^2$ then the quotient $X / \langle \alpha^2\beta^2 \rangle$ is \mbox{$G / \langle \alpha^2\beta^2 \rangle$-birationally} equivalent to an Iskovskikh surface $S$ by Lemma \ref{DP2type1}. The element~$\zeta$ does not preserve hyperplane sections $\lambda x = \mu y$ thus $\zeta$ faithfully acts on the base of the conic bundle $S \rightarrow \Pro^1_{\ka}$. Therefore the quotient $X / G \approx S / \left( G / \langle \alpha^2\beta^2 \rangle\right)$ is birationally equivalent to a surface $Y$ such that $K_Y^2 \geqslant 5$ by Theorem \ref{Cbundle}.

Assume that the group $G$ does not contain $\alpha^2\beta^2$ and $\gamma$, since otherwise $X / G$ is birationally equivalent to a surface $Y$ such that $K_Y^2 \geqslant 5$. By Theorem \ref{DP2groupclass} one has~$\Aut(\XX) \cong \CG_2 \times \widetilde{\AG}_4$, where $\widetilde{\AG}_4$ is a central extension of $\AG_4$, and $|\widetilde{\AG}_4| = 48$. Consider the natural homomorphism $G \rightarrow \AG_4$. The image of $G$ under this homomorphism is~$\AG_4$ or~$\CG_3$.

In the first case this homomorphism has nontrivial kernel since in the group $\widetilde{\AG}_4$ there are no subgroups isomorphic to $\AG_4$. Note that the kernel of the homomorphism $\widetilde{\AG_4} \rightarrow \AG_4$ is generated by $\alpha\beta$ and $\gamma$. Thus the only possibility for a nontrivial kernel is $\langle \alpha^2\beta^2\gamma\rangle$. In this case the group $G \cap \langle \alpha \beta, \alpha \beta^3, \delta \rangle$ is isomorphic to $\CG_2^2$. But any subgroup of $\langle \alpha \beta, \alpha \beta^3, \delta \rangle$ isomorphic to $\CG_2^2$ contains the element $\alpha^2 \beta^2$. Thus this case is impossible.

If the image of $G$ in $\AG_4$ is $\CG_3$, then, as in the previous case, one can see that the kernel of the homomorphism $G \rightarrow \CG_3$ is either trivial or isomorphic to $\langle \alpha^2 \beta^2 \gamma \rangle$. Thus $G$ is $\langle \zeta \rangle \cong \CG_3$ or $\langle \zeta, \alpha^2\beta^2\gamma \rangle \cong \CG_6$, respectively. In these cases~$X / G$ is birationally equivalent to a surface~$Y$ such that $K_Y^2 \geqslant 5$ by Lemma \ref{DP2type3}.

\end{proof}

For the remaining cases we need the following lemma.

\begin{lemma}
\label{DP2S3}
Let a finite group $G \cong \SG_3$ generated by elements of type $1$ act on a del Pezzo surface $X$ of degree $2$. Then the quotient $X / G$ is birationally equivalent to a surface~$Y$ such that $K_Y^2 \geqslant 5$.
\end{lemma}

\begin{proof}

By Lemma~\ref{DP2type4} the quotient $X / \CG_3$ is $\CG_2$-birationally equivalent to a del Pezzo surface~$S$ of degree $4$. Any element of order $2$ in $\SG_3$ has a curve of fixed points on $X$, thus the element of order $2$ in $\SG_3 / \CG_3 \cong \CG_2$ has a curve of fixed points on $S$. Therefore the quotient $X / G \approx S / \CG_2$ is birationally equivalent to a surface $Y$ such that $K_Y^2 \geqslant 5$ by~\cite[Lemmas 5.3 and 5.7]{Tr17}.

\end{proof}

\begin{lemma}
\label{DP23groups}
If $\XX$ is a del Pezzo surface of degree $2$ of type $II$, $IV$ or $IX$ (see Theorem~\ref{DP2groupclass}) then the group $\Aut(\XX)$ is a subgroup of $\langle \alpha, \beta, \phi, \delta, \gamma \rangle$, where $\alpha$, $\beta$, $\delta$ and $\gamma$ are defined in Notation \ref{DP2notation} and
$$
\phi: (x : y : z : t) \mapsto (z : x : y: t).
$$ 
\end{lemma}

\begin{proof}
One can easily check that $\Aut(\XX)$ is:
\begin{itemize}
\item $\langle \alpha, \beta, \phi, \delta, \gamma \rangle$ for the case $II$;
\item $\langle \alpha^2, \beta^2, \phi, \delta, \gamma \rangle$ for the case $IV$;
\item $\langle \phi, \delta, \gamma \rangle$ for the case $IX$.
\end{itemize}
\end{proof}

\begin{lemma}
\label{DP2odd2}
Let a finite group $G$ of order $2^k \cdot 3$ act on a del Pezzo surface $X$ of degree~$2$ and $X$ be of type $I$, $II$, $IV$ or $IX$. If the group $G$ is not listed in Proposition \ref{DP2} then the quotient $X / G$ is birationally equivalent to a surface~$Y$ such that $K_Y^2 \geqslant 5$.
\end{lemma}

\begin{proof}

If $X$ has type $I$ and $G$ has order $2^k \cdot 3$ then $X$ can be considered as a surface of type $IV$ by Remark \ref{DP2typeIspec}. For the other types by Lemma \ref{DP23groups} we can assume that $G$ is a subgroup of $\langle \alpha, \beta, \phi, \delta, \gamma \rangle$. All elements of order $3$ are conjugate in this group, so we can assume that $G$ contains the element $\phi$ of type $4$ in the notation of Table \ref{table2}.

Let $G_0 = G \cap \langle \alpha, \beta, \gamma \rangle$ be a subgroup. Obviously $G_0$ is normal in $G$.

If $G_0$ is nontrivial then it contains an element of order $2$. As in the proof of Lemma~\ref{DP2even}, we can assume that~$\gamma$ is not contained in $G$. Let an element $\alpha^2\gamma$ be contained in $G$. Then $\phi\alpha^2\gamma \phi^{-1} = \beta^2\gamma$ and $\phi^{-1}\alpha^2\gamma \phi = \alpha^2\beta^2 \gamma$ are contained in $G$. But the composition of these three elements is~$\gamma$. This contradicts the assumption that $\gamma$ does not lie in $G$.

If $\alpha^2$, $\beta^2$ or $\alpha^2\beta^2$ is contained in $G$ then $G$ contains a normal subgroup $N = \langle \alpha^2, \beta^2 \rangle$. Therefore $X / G$ is birationally equivalent to a surface $Y$ such that $K_Y^2 \geqslant 5$ by Lemma~\ref{DP2V4}.

Now we can assume that $G_0$ is trivial. Therefore the group $G$ is isomorphic to $\SG_3$ or $\CG_3$. The cases $G \cong \SG_3$ generated by elements of type $2$ and $G \cong \CG_3$ are listed in Proposition~\ref{DP2}. If $G \cong \SG_3$ is generated by elements of type $1$ then the quotient $X / G$ is birationally equivalent to a surface $Y$ such that $K_Y^2 \geqslant 5$ by Lemma \ref{DP2S3}.

\end{proof}

\subsection{Subgroups of $\CG_2 \times \mathrm{PSL}_2 \left( \F_7 \right)$}

In this section $X$ is a del Pezzo surface of degree~$2$ of type $I$ (see Theorem \ref{DP2groupclass}). Actually, any del Pezzo surface of degree $2$ having an automorphism of order $7$ is of type $I$ (see Table \ref{table3}). In this case the group $\Aut\left(\XX\right)$ is $\CG_2 \times \mathrm{PSL}_2 \left( \F_7 \right)$. In this subsection we study, for which subgroups $G \subset \Aut\left(\XX\right)$ the quotient $X / G$ can be non-$\ka$-rational. We assume that $G$ does not contain the Geiser involution.
 
\begin{lemma}
\label{DP2type5}
Let a finite group $G$ act on a del Pezzo surface $X$ of degree $2$ and $N \cong \CG_7$ be a normal subgroup in $G$ generated by an element of type $5$. Then the quotient $X / G$ is birationally equivalent to a surface $Y$ such that $K_Y^2 \geqslant 5$.
\end{lemma}

\begin{proof}

If $N \cong \CG_7$ is generated by an element of type $5$ then the group $N$ has three isolated fixed points $p_1 = (1: 0 : 0 : 0)$, $p_2 = (0: 1 : 0 :0)$ and $p_3 = (0 : 0 : 1 : 0)$. On the tangent spaces of $X$ at the points $p_i$ the group $N$ acts as $\langle\operatorname{diag}(\xi_7, \xi_7^3) \rangle$.

Let $C_{23}$, $C_{31}$ and $C_{12}$ be curves given by $x = 0$, $y = 0$ and $z = 0$ respectively. The curve~$C_{ij}$ is $N$-invariant, and passes through the points $p_i$ and $p_j$.

The surface $X / N$ has three $\frac{1}{7}(1,3)$ singularities. Let $f: X \rightarrow X / N$ be the quotient morphism and
$$
\pi: \widetilde{X / N} \rightarrow X / N
$$
\noindent be the minimal resolution of singularities. Then $\pi^{-1}(f(p_i))$ is a chain of curves $L_i$, $M_i$ and $N_i$ with negative self-intersection numbers such that
$$
L_i^2 = -3, \qquad M_i^2 = N_i^2 = -2, \qquad L_i \cdot M_i = M_i \cdot N_i = 1, \qquad L_i \cdot N_i = 0,
$$
\noindent see Table \ref{table1}. Each proper transform $\pi^{-1}_*f(C_{ij})$ is a $(-1)$-curve intersecting $N_i$ and passing through the intersection point of $L_j$ and $M_j$. We can $G / N$-equivariantly contract these three $(-1)$-curves, then contract the proper transforms of $N_i$, and then contract the proper transforms of $L_i$. We get a surface~$S$ such that
$$
K_S^2 = K_{\widetilde{X / N}}^2 + 9 = K_{X / N}^2 - 3 \cdot \frac{3}{7} + 9 = \frac{1}{7}K_X^2 - \frac{9}{7} + 9 = 8.
$$
By Remark \ref{geq5touse} the quotient $X / G \approx S / (G / N)$ is birationally equivalent to a surface~$Y$ such that $K_Y^2 \geqslant 5$.

\end{proof}

\begin{lemma}
\label{DP2simple}
Let a finite group $G \cong \mathrm{PSL}_2(\F_7)$ act on a del Pezzo surface $X$ of degree~$2$. Then the quotient $X / G$ is birationally equivalent to a surface~$Y$ such that $K_Y^2 = 5$.
\end{lemma}

\begin{proof}

Let $f: X \rightarrow X / G$ be the quotient morphism and
$$
\pi: \widetilde{X / G} \rightarrow X / G
$$
\noindent be the minimal resolution of the singularities.

We want to describe the set of singular points on $X / G$. If for a point $p$ on $X$ the point~$f(p)$ is singular then the stabilizer of $P$ in $G$ is nontrivial and is not generated by reflections (see Subsection 2.3). Therefore we want to find points with nontrivial stabilizers on $X$ to describe the set of singular points on $X / G$. We use the method considered in the proof of \cite[Subsection 2.1, Proposition]{Elk99}.

The group $G$ has no fixed points, since $\mathrm{PSL}_2(\F_7)$ does not have nontrivial $2$-dimensional linear representations. Therefore a stabilizer of any point on $X$ is contained in a maximal subgroup of $G$. It is well known that any maximal subgroup of $\mathrm{PSL}_2(\F_7)$ is isomorphic either to~$\SG_4$ or to $\CG_7 \rtimes \CG_3$ (see \cite[p. 3]{ATLAS}).

By Remark \ref{DP2typeIspec} we can consider a surface given by the equation
$$
x^4 + y^4 + z^4 + A(x^2y^2 + x^2z^2 + y^2z^2) + t^2 = 0
$$
\noindent to find the points with nontrivial stabilizers contained in $\SG_4$. This set of points consists of the orbits of the points $(0 : 0 : 1 : \pm \ii)$ with the stabilizer $\DG_8$, the points \mbox{$(1 : 1 : 1 : \pm \sqrt{-3-3A})$} with the stabilizer $\SG_3$, the points $(1 : -1 : 0 : \pm \sqrt{-2 - A})$ with the stabilizer $\CG_2^2$, the point $(1 : \omega : \omega^2 : 0)$ with the stabilizer $\CG_3$, and the orbits of the curves of fixed points $z = 0$ and $x = y$ with the stabilizers $\CG_2$. For all these points except the orbit of $(1 : \omega : \omega^2 : 0)$ the stabilizer is generated by reflections. The stabilizer group $\CG_3$ acts on the neighbourhood of $(1 : \omega : \omega^2 : 0)$ as $\operatorname{diag}\left(\omega, \omega \right)$. 

We can consider a surface given by the equation
$$
x^3y + y^3z + z^3x + t^2 = 0
$$
\noindent to find the points with nontrivial stabilizers contained in $\CG_7 \rtimes \CG_3$. This set of points consists of the orbits of the point $(1 : 0 : 0 : 0)$ with the stabilizer $\CG_7$ and the point $(1 : \omega : \omega^2 : 0)$ with the stabilizer $\CG_3$. Therefore the stabilizer of each point with nontrivial stabilizer not generated by reflections is either $\CG_7$, or $\CG_3$. The stabilizer group $\CG_7$ acts on the neighbourhood of $(1 : 0 : 0 : 0)$ as $\operatorname{diag}\left(\xi_7, \xi_7^3 \right)$. 

There are $28$ subgroups isomorphic to $\CG_3$ in $G$. These subgroups are generated by elements of type $4$ in the notation of Table \ref{table2}. Such an element has four fixed points and acts on the neighbourhood of two of these points as $\operatorname{diag}\left(\omega, \omega^2 \right)$, and on the neighbourhood of the two other fixed points as $\operatorname{diag}\left(\omega, \omega \right)$. Therefore there is one $\frac{1}{3}(1,1)$-singular point on $X / G$.

There are $8$ subgroups isomorphic to $\CG_7$ in $G$. These subgroups are generated by elements of type $5$ in the notation of Table \ref{table2}. Such an element has three fixed points and act on the neighbourhood of these points as $\operatorname{diag}\left(\xi_7, \xi_7^3 \right)$. Therefore there is \mbox{one $\frac{1}{7}(1,3)$-singular} point on $X / G$.

Hence the set of singular points of $X / G$ is the following: one $\frac{1}{3}(1,1)$-point, and one~$\frac{1}{7}(1,3)$-point. Each element of order $2$ in $G$ has pointwisely fixed hyperplane section, and the other elements have only isolated fixed points. By the Hurwitz formula one has
$$
K_{\widetilde{X / G}}^2 = K_{X / G}^2 - \frac{1}{3} - \frac{3}{7} = \frac{1}{168}\left(22K_X\right)^2 - \frac{16}{21} = 5.
$$
\end{proof}

\begin{lemma}
\label{DP2typeI}
Let a finite group $G$ act on a del Pezzo surface $X$ of degree $2$ and $X$ be of type $I$. If the group $G$ is not listed in Proposition \ref{DP2} then the quotient $X / G$ is birationally equivalent to a surface $Y$ such that $K_Y^2 \geqslant 5$.
\end{lemma}

\begin{proof}

If the group $G$ does not contain elements of order $7$ then $G$ is a subgroup \mbox{of $\CG_2 \times \SG_4$}, where the first factor is generated by the Geiser involution. These subgroups are considered in Lemmas \ref{DP2even} and \ref{DP2odd2}.

If $G$ contains a subgroup $\CG_7$ then by Sylow theorem there are $1$ or $8$ such subgroups in the group $G$. If there is a unique subgroup $\CG_7$ in $G$ then $X / G$ is birationally equivalent to a surface $Y$ such that $K_Y^2 \geqslant 5$ by Lemma~\ref{DP2type5}. Otherwise the group $G$ is $\mathrm{PSL}_2(\F_7)$, and $X / G$ is birationally equivalent to a surface $Y$ such that $K_Y^2 = 5$ by Lemma \ref{DP2simple}.

\end{proof}

Now we can prove Proposition \ref{DP2}.

\begin{proof}[Proof of Proposition \ref{DP2}]
If $G$ is a group that is not listed in Proposition \ref{DP2} then $X / G$ is birationally equivalent to a surface $Y$ such that $K_Y^2 \geqslant 5$ by Lemmas \ref{DP2even}, \ref{DP2odd1}, \ref{DP2odd2} and~\ref{DP2typeI}. One has $Y(\ka) \neq \varnothing$ since $X(\ka)$ is dense. Any minimal model of $Y$ is $\ka$-rational by Theorem~\ref{ratcrit}. Therefore $X / G$ is $\ka$-rational.

\end{proof}

\section{The Weyl group $\mathrm{W}(\mathrm{E}_7)$}

A del Pezzo surface $\XX$ over an algebraically closed field $\kka$ is isomorphic to a blowup of~$\Pro^2_{\kka}$ at seven points $p_1$, $\ldots$, $p_7$ in general position. Therefore the group $\Pic(\XX)$ is generated by the proper transform $L$ of the class of a line on $\Pro^2_{\kka}$, and the classes $E_1$, $\ldots$, $E_7$ of the exceptional divisors. The sublattice $K_X^{\perp}$ of classes $C$ in $\Pic(\XX)$ such that $C \cdot K_X = 0$, is generated by 
$$
L - E_1 - E_2 - E_3, \qquad E_1 - E_2, \qquad E_2 - E_3, \qquad \ldots, \qquad E_6 - E_7.
$$
This set of generators are simple roots for the root system of type $\mathrm{E}_7$. Therefore any group acting on the Picard lattice $\Pic(\XX)$ and preserving the intersection form is a subgroup of the Weyl group $\mathrm{W}(\mathrm{E}_7)$. Moreover, if a finite group $G$ acts on a del Pezzo surface~$X$ of degree $2$ then there is an embedding $G \hookrightarrow \mathrm{W}(\mathrm{E}_7)$ (see \cite[Lemma 6.2]{DI1}). For convenience we will identify the group $G$ with its image in $\mathrm{W}(\mathrm{E}_7)$. Also we denote the image of the group $\Gal\left(\kka / \ka\right)$ in $\mathrm{W}(\mathrm{E}_7)$ by $\Gamma$. The groups $G$ and $\Gamma$ commute.

Note that one can choose seven disjoint $(-1)$-curves on $\XX$ corresponding to a blowup $\XX \rightarrow \Pro^2_{\kka}$ in many ways. Therefore the embeddings of $G$ and $\Gamma$ into $\mathrm{W}(\mathrm{E}_7)$ are defined up to conjugacy.

In this section we study some properties of the group $\mathrm{W}(\mathrm{E}_7)$ and its subgroups. First we prove the following lemma.

\begin{lemma}
\label{DP2type2nonrat}
Let a finite group $G \cong \CG_2$ act on a del Pezzo surface $X$ of degree $2$. Suppose that the group $G$~is generated by an element of type $2$ in the notation of Table \ref{table2}, \mbox{and $\rho(X)^G = 1$}. Then~$X$ is non-$\ka$-rational.
\end{lemma}

\begin{proof}
Let $V$ be a vector space $K_X^{\perp} \otimes \mathbb{Q} \subset \Pic(\XX) \otimes \mathbb{Q}$, and $g$ be the generator of $G$. The element $g$ has order $2$, and eigenvalues of its action on $V$ are $\pm 1$. The Geiser involution~$\gamma$ acts as multiplication by $-1$ on $V$. Therefore one has $V = V^{\langle g \rangle} \oplus V^{\langle g\gamma \rangle}$.

One has $V^{G \times \Gamma} = 0$, since $\rho(X)^G = 1$. Thus $V^{\Gamma} \subset V^{\langle g\gamma \rangle}$. Let $Y$ be a minimal model of $X$. If $X$ is $\ka$-rational then $Y$ is a minimal del Pezzo surface of degree $5$ or greater by Theorem \ref{ratcrit}. For the birational morphism $X \rightarrow Y$, the classes of contracted divisors lie in $V^{\Gamma}$, therefore the contraction $X \rightarrow Y$ is $\langle g\gamma \rangle$-equivariant.

The element $g\gamma$ has type $1$ in the notation of Table \ref{table2}, and thus it pointwisely fixes a hyperplane section of $X$, that is an elliptic curve. Therefore the element $g\gamma$ should pointwisely fix an elliptic curve on $Y$. But it is well-known that for an involution $\iota$ acting on a del Pezzo surface of degree~$5$ or greater the set of fixed points can consist only of isolated fixed points and curves of genus $0$, since there exists a $\iota$-equivariant birational map to $\Pro^2_{\ka}$. The obtained contradiction shows that $X$ is non-$\ka$-rational.
\end{proof}

\begin{remark}
\label{DP1C2nonmin}
Note that the same method works for del Pezzo surfaces of degree $1$. Indeed let an element $g$ of order $2$ act on a del Pezzo surface $X$ of degree $1$, and suppose that~$\rho(X)^{\langle g \rangle} = 1$. If for the Bertini involution $\beta$ the element $g\beta$ pointwisely fixes a hyperplane section then $X$ is non-$\ka$-rational.
\end{remark}

Now we want to study some properties of elements of type $4$ in the notation of Table~\ref{table2}. Throughout this section we use the following notation.

\begin{notation}
\label{DP_1curves}
Let $X$ be a del Pezzo surface of degree $2$. Then $\XX$ can be realized as a blowup~\mbox{$f: \XX \rightarrow \Pro^2_{\kka}$} at $7$ points $p_1$, $\ldots$, $p_7$ in general position. Put $E_i = f^{-1}(p_i)$ \mbox{and $L = f^*l$}, where~$l$ is the class of a line on $\Pro^2_{\kka}$. One has
$$
-K_{\XX} \sim 3L - \sum \limits_{i=1}^7 E_i.
$$
The set of $(-1)$-curves on $\XX$ consists of $E_i$, the proper transforms \mbox{$L_{ij} \sim L - E_i - E_j$} of the lines passing through a pair of points $p_i$ and $p_j$, the proper transforms
$$
Q_{ij} \sim 2L + E_i + E_j - \sum \limits_{k = 1}^7 E_k
$$
\noindent of the conics passing through all the points of the blowup except $p_i$ and $p_j$, and the proper transforms
$$
C_i \sim 3L - E_i - \sum \limits_{k = 1}^7 E_k
$$
of the rational cubic curves passing through all the points of the blowup, and having a singularity at $p_i$.

In this notation one has:
$$
E_i \cdot E_j = 0; \qquad E_i \cdot L_{ij} = 1; \qquad E_i \cdot L_{jk} = 0; \qquad E_i \cdot Q_{ij} = 0; \qquad E_i \cdot Q_{jk} = 1;
$$
$$
E_i \cdot C_i = 2; \qquad E_i \cdot C_j = 1; \qquad L_{ij} \cdot L_{ik} = 0; \qquad L_{ij} \cdot L_{kl} = 1; \qquad L_{ij} \cdot Q_{ij} = 2;
$$
$$
L_{ij} \cdot Q_{ik} = 1; \qquad L_{ij} \cdot Q_{kl} = 0; \qquad L_{ij} \cdot C_i = 0; \qquad L_{ij} \cdot C_k = 1;
$$
$$
Q_{ij} \cdot Q_{ik} = 0; \qquad Q_{ij} \cdot Q_{kl} = 1; \qquad Q_{ij} \cdot C_i = 1; \qquad Q_{ij} \cdot C_k = 0,
$$
\noindent where $i$, $j$, $k$ and $l$ are different numbers from the set $\{1, 2, 3, 4, 5, 6, 7\}$.
\end{notation}

One can see that the Geiser involution $\gamma$ maps $E_i$ to $C_i$, and maps $L_{ij}$ to $Q_{ij}$. Also there is a (non-normal) subgroup $\SG_7 \subset \mathrm{W}(\mathrm{E}_7)$ that acts on the set of $(-1)$-curves in the following way: for a given permutation $\sigma \in \SG_7$ one has
$$
\sigma(E_i) = E_{\sigma(i)}; \qquad \sigma(L_{ij}) = L_{\sigma(i)\sigma(j)}; \qquad \sigma(Q_{ij}) = Q_{\sigma(i)\sigma(j)}; \qquad \sigma(C_i) = C_{\sigma(i)}.
$$

\begin{lemma}
\label{Eltype4}
An element of type $4$ in the notation of Table \ref{table2} is conjugate \mbox{to $(123)(456) \in \SG_7$} in~$\mathrm{W}(\mathrm{E}_7)$.
\end{lemma}

\begin{proof}

By Theorem \ref{DP2cyclicclass} a del Pezzo surface $\XX$ can be given by the equation
$$
(x^3 + y^3)z + Ax^2y^2 + Bxyz^2 + z^4 + t^2 = 0
$$
\noindent in $\Pro^2_{\kka}(1 : 1 : 1 : 2)$, and an element $g$ of type $4$ acts as
$$
(x : y : z : t) \mapsto ( \omega x : \omega^2 y : z : t).
$$

The $g$-invariant plane section $z = 0$ is reducible and consists of two $(-1)$-curves. Denote one of these curves by $E_7$ and consider the contraction of $E_7$. We get a cubic surface $\overline{Y}$, and the element $g$ acts on $\overline{Y}$ and has only isolated fixed points, since $g$ has only isolated fixed points on $\XX$. In this case the action of~$g$ on~$\overline{Y}$ is conjugate to $(123)(456) \in \SG_6 \subset \mathrm{W}(\mathrm{E}_6)$ by \cite[Lemma 4.1]{Tr16b}. Therefore the action of $g$ on $\XX$ is conjugate \mbox{to $(123)(456) \in \SG_7 \subset \mathrm{W}(\mathrm{E}_7)$}.

\end{proof}

In the rest of this section let $G$ be the group $\langle (123)(456) \rangle \cong \CG_3$. The group $\Gamma \subset \mathrm{W}(\mathrm{E}_7)$ commutes with $G$. Thus~$\Gamma$ is always a subgroup of the centralizer $C(G)$ of $G$ in $\mathrm{W}(\mathrm{E}_7)$. To describe the group~$C(G)$ we use the following notation:
$$
a = (123), \qquad b = (456), \qquad c = (14)(25)(36),
$$
\noindent and $r$ and $s$ are elements in $\mathrm{W}(\mathrm{E}_7)$ of order $3$ and $2$ respectively such that
$$
s(E_7) = E_7; \qquad s(E_i) = Q_{i7}\,\,\textrm{if}\,\,1 \leqslant i \leqslant 6;
$$
$$
s(L_{ij}) = L_{ij}\,\,\textrm{if}\,\,1 \leqslant i < j \leqslant 6; \qquad s(L_{i7}) = C_{i}\,\,\textrm{if}\,\,1 \leqslant i \leqslant 6
$$
$$
r(E_7) = E_7; \qquad r(E_i) = Q_{i7} \,\,\textrm{if}\,\,1 \leqslant i \leqslant 3; \qquad r^2(E_i) = Q_{i7} \,\,\textrm{if}\,\,4 \leqslant i \leqslant 6;
$$
$$
r^2(E_i) = L_{jk} \,\,\textrm{if}\,\,i, j\,\,\textrm{and}\,\,k\,\,\textrm{are different numbers from the set}\,\,\{1, 2, 3\};
$$
$$
r(E_i) = L_{jk} \,\,\textrm{if}\,\,i, j\,\,\textrm{and}\,\,k\,\,\textrm{are different numbers from the set}\,\,\{4, 5, 6\}.
$$
This notation is introduced in \cite[Section 4]{Tr16b}, and we use some results of that paper. 

\begin{proposition}
\label{type4center}
The centralizer of the group $G \cong \CG_3$ generated by the element $ab$ in $\mathrm{W}(\mathrm{E}_7)$ is a subgroup
$$
C(G) = \langle a, b, cs, r, s, \gamma \rangle \cong \left( \CG_3^2 \rtimes \CG_2 \right) \times \SG_3 \times \CG_2,
$$
\noindent where the first factor is generated by $a$, $b$ and $cs$, the second factor is generated by $r$ and~$s$, and the third factor is generated by $\gamma$.
\end{proposition}

\begin{proof}
In $\mathrm{W}(\mathrm{E}_7)$ the element $ab$ has two invariant classes of $(-1)$-curves: $E_7$ and $C_7$. These classes are permuted by the Geiser involution $\gamma$ that commutes with $ab$. Therefore one has \mbox{$C(G) = C_0(G) \times \langle \gamma \rangle$}, where the group $C_0(G)$ is the stabilizer of $E_7$ in $C(G)$. The group $C_0(G)$ faithfully acts on the sublattice $E_7^{\perp} \cong \langle L, E_1, \ldots, E_6 \rangle$. Therefore $C_0(G)$ is a subgroup of $\mathrm{W}(\mathrm{E}_6) \subset \mathrm{W}(\mathrm{E}_7)$, and any element of $C_0(G)$ commutes with $ab \in \mathrm{W}(\mathrm{E}_6)$. The group $C_0(G)$ is
$$
\langle a, b, cs, r, s \rangle \cong \left( \CG_3^2 \rtimes \CG_2 \right) \times \SG_3
$$
\noindent by \cite[Proposition 4.5]{Tr16b}. Therefore
$$
C(G) = \langle a, b, cs, r, s, \gamma \rangle \cong \left( \CG_3^2 \rtimes \CG_2 \right) \times \SG_3 \times \CG_2.
$$
\end{proof}

We want to find subgroups of $C(G)$ such that $\rho(X)^G = 1$ and there exists a birational morphism $X \rightarrow Y$, where $K_Y^2 \geqslant 5$. In Section $6$ we use this to construct examples of quotients of $\ka$-rational del Pezzo surfaces of degree $2$.

In the following two lemmas and proposition the reader should pay attention to difference between the notions of minimal surface and $G$-minimal surface, $\rho(X)$ and $\rho(X)^G$, morphism and \mbox{$G$-equivariant} morphism.

\begin{lemma}
\label{type4min2nonrational}
If $s\gamma \in \Gamma$ then $X$ is minimal.
\end{lemma}

\begin{proof}
Note that $(-1)$-curves $L_{16}$, $L_{26}$, $L_{36}$, $L_{46}$, $L_{56}$ and $E_7$ are disjoint and $s$-invariant. Therefore the action of $s$ on $K_X^{\perp} \otimes \mathbb{Q}$ has one eigenvalue $-1$ and six eigenvalues $1$. The Geiser involution $\gamma$ acts on $K_X^{\perp} \otimes \mathbb{Q}$ by the multiplication on $-1$, thus the action of $s\gamma$ on $K_X^{\perp} \otimes \mathbb{Q}$ has one eigenvalue $1$ and six eigenvalues $-1$, and $\rho(\XX)^{\langle s\gamma \rangle} = 2$. One can check that $\Pic(\XX)^{\langle s\gamma \rangle}$ is generated by $-K_X$ and $L - E_7$. The linear system $|L - E_7|$ gives a structure of a conic bundle on $\XX$, therefore $X$ is minimal by Theorem~\ref{MinCB}.
\end{proof}

\begin{lemma}
\label{type4min3nonrational}
If the group $\Gamma$ is $\langle abr \rangle$, $\langle a^2br \rangle$, $\langle ab, r \rangle$ or $\langle a^2b, r \rangle$ then $\rho(X) = 2$ and $\Pic(X)$ is generated by $K_X$ and $E_7$.
\end{lemma}

\begin{proof}
This lemma directly follows from \cite[Lemma 4.10]{Tr16b}.
\end{proof}

\begin{proposition}
\label{type4minrational}
Let $X$ be a del Pezzo surface of degree $2$ such that the group \mbox{$G = \langle ab \rangle \cong \CG_3$} acts on~$X$. There exists a birational morphism $X \rightarrow Y$ such that $K_Y^2 \geqslant 5$, and $\rho(X)^G = 1$ if and only if $\Gamma$ is conjugate to $\langle r, cs\gamma \rangle$, $\langle r, c\gamma \rangle$ or $\langle r, c\gamma, s \rangle$.
\end{proposition}

\begin{proof}
Assume that there exists a birational morphism $X \rightarrow Y$ such that $K_Y^2 \geqslant 5$ and~$\rho(X)^G = 1$. In particular, $X$ is not minimal.

Let us consider a $2$-Sylow subgroup $\Gamma_2$ of $\Gamma$. We may assume that $\Gamma_2$ is a subgroup of a~$2$-Sylow subgroup of $C(G)$, that is $\langle c, s, \gamma \rangle \cong \CG_3^2$ (see Proposition \ref{type4center}). If $\Gamma_2$ contains $\gamma$ or $s\gamma$ then $X$ is minimal by Lemma \ref{type4min2nonrational}, and we have a contradiction. Moreover, the curve~$E_7$ is $G$-invariant and $\langle c, s \rangle$-invariant. Therefore $\Gamma_2$ is $\langle c\gamma \rangle$, $\langle cs\gamma \rangle$ or $\langle c\gamma, s \rangle$.

By Lemma \ref{type4min3nonrational} the group $\Gamma$ does not contain $abr$, $a^2br$, $ab^2r$ and $a^2b^2r$. If $\Gamma$ contains~$ar$ then $\Gamma$ contains $br$ and $abr^2$, since $cac^{-1} = b$ and $\Gamma$ contains $c\gamma$ or $cs\gamma$. Thus a $3$-Sylow subgroup of $\Gamma$ is either a subgroup of $\langle a, b \rangle$ or is $\langle r \rangle \cong \CG_3$.

Note that $\Pic(\XX)^{\langle ab \rangle} = \Pic(\XX)^{\langle a^2b \rangle}$ is generated by $-K_X$, $E_1 + E_2 + E_3$, $E_4 + E_5 + E_6$ and $E_7$. Therefore $\Gamma$ does not contain $ab$ and $a^2b$, since otherwise $\rho(X) = \rho(X)^G = 1$. Moreover, if $\Gamma$ contains $a$ then $\Gamma$ contains $b$ and $ab$, since $cac^{-1} = b$ and $\Gamma$ contains $c\gamma$ or~$cs\gamma$.

Note that if $\Gamma \subset \langle c\gamma, s \rangle$ then the set of disjoint curves $L_{12}$, $L_{13}$, $L_{23}$, $Q_{45}$, $Q_{46}$ and $Q_{56}$ is $G \times \Gamma$-invariant, and $\rho(X)^G > 1$. Therefore $\Gamma$ is $\langle r, cs\gamma \rangle$, $\langle r, c\gamma \rangle$ or $\langle r, c\gamma, s \rangle$.

Now we show that the groups $\langle r, cs\gamma \rangle$, $\langle r, c\gamma \rangle$ and $\langle r, c\gamma, s \rangle$ satisfy the conditions of Proposition \ref{type4minrational}.

The groups $\langle r, cs\gamma \rangle$ and $\langle r, c\gamma \rangle$ are subgroups of the group $\langle r, c\gamma, s \rangle$. One can \mbox{$\langle r, c\gamma, s \rangle$-equivariantly} contract $(-1)$-curves $L_{15}$, $Q_{24}$, $L_{16}$ and $Q_{34}$ and get a del Pezzo surface $Y$ of degree $6$.

By Lemma \ref{type4min3nonrational} one has $\rho(\XX)^{\langle ab, r \rangle} = 2$ and $\Pic(\XX)^{\langle ab, r \rangle}$ is generated by $K_X$ and $E_7$. The elements $cs\gamma$ and $c\gamma$ do not preserve the curve $E_7$. Therefore for the groups $\langle ab, r, cs\gamma \rangle$, $\langle ab, r, c\gamma \rangle$ or $\langle ab, r, c\gamma, s \rangle$ one has $\rho(X)^G = 1$.

\end{proof}

\section{Examples}

In this section we construct explicit examples of quotients of del Pezzo surfaces $X$ of degree $2$ by finite groups $G$ listed in Proposition \ref{DP2} such that $\rho(X)^G = 1$, $X(\ka) \neq \varnothing$ and there is a smooth $\ka$-point on $X / G$. If $G$ is trivial then $X$ is non-$\ka$-rational by Theorem~\ref{ratcrit}, and if $G$ is of type $(3)$ then $X$ is non-$\ka$-rational by Lemma \ref{DP2type2nonrat} and $X / G$ is non-$\ka$-rational by Remark \ref{DP2type2min}. For the other groups listed in Proposition \ref{DP2} we show that each of the four possibilities of $\ka$-rationality of $X$ and $X / G$ is realized for certain $\ka$: the surface $X$ can be $\ka$-rational and $X / G$ can be non-$\ka$-rational, $X$ can be non-$\ka$-rational and $X / G$ can be non-$\ka$-rational, $X$ can be $\ka$-rational and $X / G$ can be $\ka$-rational, $X$ can be non-$\ka$-rational and $X / G$ can be $\ka$-rational. For each type of group, $\ka$-rational and non-$\ka$-rational $X$, \mbox{$\ka$-rational} and non-$\ka$-rational $X / G$ we give a reference to the corresponding example in the following table.

\vbox{
\begin{center}
\begin{longtable}{|c|c|c|c|c|}
\caption[]{\label{table4} examples of quotients}\endhead
\hline
Type of $G$ & $X$ rat., $X / G$ rat. & $X$ rat., $X / G$ not & $X$ not, $X / G$ rat. & $X$ not, $X / G$ not  \\
\hline
$\mathrm{id}$, type $(1)$ & Impossible & Impossible & Impossible & Example \ref{example5} \\
\hline
$\CG_2$, type $(2)$ & Example \ref{example4} & Example \ref{example1} & Example \ref{example8} & Example \ref{example5} \\
\hline
$\CG_2$, type $(3)$ & Impossible & Impossible & Impossible & Example \ref{example5} \\
\hline
$\SG_3$, type $(4)$ & Example \ref{rattorat} & Example \ref{rattonrat} & Example \ref{nrattorat} & Example \ref{nrattonrat} \\
\hline
$\CG_4$, type $(5)$ & Example \ref{example2} & Example \ref{example1} & Example \ref{example8} & Example \ref{example5} \\
\hline
$\CG_4$, type $(6)$ & Example \ref{example2} & Example \ref{example1} & Example \ref{example7} & Example \ref{example5} \\
\hline
$\CG_2^2$, type $(7)$ & Example \ref{example1} & Example \ref{example4} & Example \ref{example5} & Example \ref{example7} \\
\hline
$\SG_3$, type $(8)$ & Example \ref{rattorat} & Example \ref{rattonrat} & Example \ref{nrattorat} & Example \ref{nrattonrat} \\
\hline
$\DG_8$, type $(9)$ & Example \ref{example1} & Example \ref{example3} & Example \ref{example5} & Example \ref{example6}  \\
\hline
$Q_8$, type $(10)$ & Example \ref{example0} & Example \ref{example1} & Example \ref{example8} & Example \ref{example5} \\
\hline
$Q_8$, type $(11)$ & Example \ref{example2} & Example \ref{example1} & Example \ref{example6} & Example \ref{example5} \\
\hline

\end{longtable}
\end{center}}

\subsection{2-groups}

Let $X$ be a del Pezzo surface of degree $2$ and $G$ be a finite subgroup of~$\Aut_{\ka}(X)$. In this subsection for the groups of types $(2)$, $(5)$, $(6)$, $(7)$, $(9)$, $(10)$, $(11)$ in the notation of Proposition \ref{DP2}, we construct examples of \mbox{$\ka$-rational} and \mbox{non-$\ka$-rational} quotients of $\ka$-rational and non-$\ka$-rational del Pezzo surfaces of degree~$2$ such \mbox{that $\rho(X)^G = 1$}. As a by-product for the groups of types $(1)$ and $(3)$ we construct examples of non-$\ka$-rational quotients of non-$\ka$-rational del Pezzo surfaces of degree $2$ such that $\rho(X)^G = 1$.

Assume that the field $\ka$ contains $\ii$. Let $X$ be a del Pezzo surface of degree $2$ given in~$\Pro_{\ka}(1 : 1 : 1 : 2)$ by the equation
\begin{equation}
\label{exequation}
Ax^4 + 2Bx^2y^2 + Ay^4 + Cz^4 - t^2 = 0,
\end{equation}
\noindent where
$$
A = w^2\eta = u^2\sigma + v^2 \tau, \qquad B = u^2\sigma - v^2 \tau, \qquad C = q^2 \sigma \tau \eta,
$$
\noindent and $u$, $v$, $w$, $\sigma$, $\tau$ and $\eta$ are nonzero elements of $\ka$. Note that if $\sigma = \tau$, $\sigma = \eta$ or $\tau = \eta$ then such triple $\{ u, v, w \}$ exists, since there is a $\ka$-point on a smooth conic $w^2\eta = u^2\sigma + v^2 \tau$ \mbox{in $\Pro^2_{\ka} = \mathrm{Proj}[u, v, w]$}. Otherwise, we require additional conditions on the field $\ka$.

A finite group $\left((\CG_4 \times \CG_2) \rtimes \CG_2\right) \times \CG_2$ generated by
$$
\alpha\beta: (x : y : z : t) \mapsto (\ii x : \ii y : z : t); \qquad \alpha\beta^3: (x : y : z : t) \mapsto (\ii x : -\ii y : z : t);
$$
$$
\delta: (x : y : z : t) \mapsto (y : x : z : t); \qquad \gamma: (x : y : z : t) \mapsto (x : y : z : -t),
$$
\noindent acts on $X$. We assume that $G$ is a subgroup in $\left((\CG_4 \times \CG_2) \rtimes \CG_2\right) \times \CG_2$ of a type listed in Proposition \ref{DP2}, and $G$ contains a normal subgroup $N = \langle \alpha^2\beta^2 \rangle$. In the notation of Proposition \ref{DP2} the group $G$ has type $(2)$, $(5)$, $(6)$, $(7)$, $(9)$, $(10)$ or $(11)$.

We want to prove the following proposition.

\begin{proposition}
\label{DP22example}
Assume that the field $\ka$ contains $\ii$, a group $G$ acts on a del Pezzo surface $X$ of degree $2$ and has type $(2)$, $(5)$, $(6)$, $(7)$, $(9)$, $(10)$ or $(11)$ of Proposition \ref{DP2}, there is a smooth $\ka$-point on $X / G$ and $\rho(X)^G = 1$. Under these conditions the following holds.

\begin{itemize}
\item For any $\ka$ and $G$ of type $(10)$ there exists an example of a $\ka$-rational surface $X$ such that $X / G$ is $\ka$-rational.

\item If the field $\ka$ contains an element $\mu$ such that $\Gal\left(\ka(\sqrt{\mu}) / \ka\right) \cong \CG_2$, then for $G$ of type $(5)$, $(6)$, $(7)$, $(9)$ or $(11)$ there exists an example of a $\ka$-rational surface $X$ such that $X / G$ is $\ka$-rational; for $G$ of type $(2)$, $(5)$, $(6)$, $(9)$, $(10)$ or $(11)$ there exists an example of a $\ka$-rational surface $X$ such that $X / G$ is non-$\ka$-rational; for~$G$ of type $(7)$, $(9)$ or $(11)$ there exists an example of a non-$\ka$-rational surface $X$ such that $X / G$ is $\ka$-rational; for $G$ of type~$(2)$, $(5)$, $(6)$, $(9)$, $(10)$ or $(11)$ there exists an example of a non-$\ka$-rational surface $X$ such that $X / G$ is non-$\ka$-rational. 

\item If the field $\ka$ contains elements $\mu$, $\nu$ such that $\Gal\left(\ka\left(\sqrt{\mu}, \sqrt{\nu}\right) / \ka\right) \cong \CG_2^2$ and the conic $\mu u^2 + \nu v^2 = \mu\nu w^2$ in $\Pro^2_{\ka}$ has a $\ka$-point, then for $G$ of type $(2)$ there exists an example of a $\ka$-rational surface $X$ such that $X / G$ is $\ka$-rational; for~$G$ of type $(7)$ there exists an example of a $\ka$-rational surface $X$ such that $X / G$ is non-$\ka$-rational; for $G$ of type $(2)$, $(5)$, $(6)$ or $(10)$ there exists an example of a non-$\ka$-rational surface $X$ such that $X / G$ is $\ka$-rational; for $G$ of type $(7)$ there exists an example of a non-$\ka$-rational surface $X$ such that $X / G$ is non-$\ka$-rational.

\end{itemize}

\end{proposition}

\begin{remark}
The assumption that the field $\ka$ has a quadratic extension is necessary for non-$\ka$-rational quotients, since otherwise constructed conic bundles are not minimal.

Moreover, in some cases the field $\ka$ should have an extension of degree $4$. For example, in the conditions of Proposition \ref{DP22example} quotients by the group $N$ can be $\ka$-rational only if the field $\ka$ has an extension of degree $4$. In this case the isolated fixed points of $N$ are defined over $\ka$ by Corollary \ref{DP2evencrit}. All $(-1)$-curves passing through an isolated $N$-fixed point are $N$-invariant. Therefore $\rho(X)^N = 1$ only if the four $(-1)$-curves passing through an $N$-fixed point are transitively permuted by the group $\Gal\left( \kka / \ka \right)$.

Note that if the field $\ka$ contains elements $\mu$, $\nu$ such that $\Gal\left(\ka\left(\sqrt{\mu}, \sqrt{\nu}\right) / \ka\right) \cong \CG_2^2$ then in some cases the conic $\mu u^2 + \nu v^2 = \mu\nu w^2$ in $\Pro^2_{\ka}$ has a $\ka$-point. For example, for the field $\mathbb{Q}(\ii)$ one has $\Gal\left(\ka\left(\sqrt{2}, \sqrt{3}\right) / \ka\right) \cong \CG_2^2$, and the conic $2u^2 + 3v^2 = 6w^2$ has a $\ka$-point~$(3 : 2\ii : 1)$. The author does not know an example of a field having an extension of degree $4$ with the Galois group $\CG_2^2$ such that if for any elements $\mu \in \ka$, $\nu \in \ka$ one has $\Gal\left(\ka\left(\sqrt{\mu}, \sqrt{\nu}\right) / \ka\right) \cong \CG_2^2$ then the conic $\mu u^2 + \nu v^2 = \mu\nu w^2$ in $\Pro^2_{\ka}$ does not have a $\ka$-point.
\end{remark}

To prove Proposition \ref{DP22example} for each considered type of $G$ in Examples \ref{example0}--\ref{example8} we construct \mbox{$\ka$-rational} and non-$\ka$-rational surfaces $X$ with \mbox{$\ka$-rational} and non-$\ka$-rational quotients $X / G$ (see Table \ref{table4} for more precise references).

One can find explicit equations of all $56$ lines on $X$. For convenience we refer to these lines in the following way:

~

\noindent eight \textit{$\theta$-lines}, given by
$$
t = \pm q\sqrt{\sigma \tau \eta} z^2, \qquad x = ky, \qquad k = \frac{\pm v \sqrt{\tau} \pm \ii u \sqrt{\sigma}}{w\sqrt{\eta}};
$$
\noindent sixteen \textit{$\eta$-lines}, given by
$$
t = \pm w\sqrt{\eta} \left( x^2 + \frac{B}{A} y^2 \right), \qquad z = ky, \qquad k^2 = \pm\frac{2\ii uv}{wq \eta};
$$
$$
t = \pm w\sqrt{\eta} \left( \frac{B}{A}x^2 + y^2 \right), \qquad z = kx, \qquad k^2 = \pm\frac{2\ii uv}{wq \eta};
$$
\noindent sixteen \textit{$\sigma$-lines}, given by
$$
t = \pm \frac{1}{u\sqrt{\sigma}} \left( A(x^2 + y^2) + (A - B) xy \right), \qquad z = k(x + y), \qquad k^2 = \pm\frac{vw}{uq \sigma};
$$
$$
t = \pm \frac{1}{u\sqrt{\sigma}} \left( A(x^2 + y^2) + (B - A) xy \right), \qquad z = k(x - y), \qquad k^2 = \pm\frac{vw}{uq \sigma};
$$
\noindent sixteen \textit{$\tau$-lines}, given by
$$
t = \pm \frac{1}{v\sqrt{\tau}} \left( A(x^2 - y^2) + \ii (A + B) xy \right), \qquad z = k(x + \ii y), \qquad k^2 = \pm\frac{uw}{vq \tau};
$$
$$
t = \pm \frac{1}{v\sqrt{\tau}} \left( A(x^2 - y^2) - \ii (A + B) xy \right), \qquad z = k(x - \ii y), \qquad k^2 = \pm\frac{uw}{vq \tau}.
$$

Note that each of the sets of $\theta$-lines, $\eta$-lines, $\sigma$-lines, $\tau$-lines is invariant under the action of $G$ and $\Gal\left( \kka / \ka \right)$. Moreover, the action of any element of $\Gal\left( \kka / \ka\right)$ on a set of $\eta$-lines, $\sigma$-lines or $\tau$-lines is trivial or coincides with the action of $\alpha^2\beta^2$, $\alpha^2\beta^2\gamma$ or $\gamma$.

\begin{remark}
\label{a2b2lines}
One can easily see that each $\theta$-line is $\alpha^2\beta^2$-invariant, and for each line $E$ from the sets of $\eta$-lines, $\sigma$-lines, $\tau$-lines one has $E \cdot \alpha^2\beta^2 E = 1$.
\end{remark}

We use the following definition for convenience.

\begin{definition}
A $G \times \Gal\left( \kka / \ka \right)$-invariant set $\Sigma$ of lines on $\XX$ is called \textit{minimal} (resp. \textit{$G$-minimal}) if a $\Gal\left( \kka / \ka \right)$-orbit (resp. $G \times \Gal\left( \kka / \ka \right)$-orbit) of each line in $\Sigma$ has \mbox{class $-nK_X$}.
\end{definition}

In particular, if $\Sigma$ is a minimal (resp. $G$-minimal) set of lines, then for any line $D$ in~$\Sigma$ one can not contract the $\Gal\left( \kka / \ka \right)$-orbit (resp. $G \times \Gal\left( \kka / \ka \right)$-orbit) of $D$.

Obviously, $\rho(X) = 1$ (resp. $\rho(X)^G = 1$) if and only if each of the sets of $\theta$-lines, $\eta$-lines, $\sigma$-lines, $\tau$-lines is minimal (resp. $G$-minimal). But we need the following lemma. 

\begin{lemma}
\label{linesmincrit}
One has $\rho(X) = 1$ (resp. $\rho(X)^G = 1$) if and only if each of the sets of $\eta$-lines, $\sigma$-lines, $\tau$-lines is minimal (resp. $G$-minimal).
\end{lemma}

\begin{proof}
Assume that each of the sets of $\eta$-lines, $\sigma$-lines, $\tau$-lines is minimal. If there exists a (resp. $G$-equivariant) birational morphism $f: X \rightarrow Z$, and $E$ is the exceptional divisor of $f$ then $E$ meets each $\eta$-line, $\sigma$-line and $\tau$-line, since $-K_X \cdot E > 0$. Therefore the images of $\eta$-lines, $\sigma$-lines and $\tau$-lines on $Z$ have non-negative self-intersection. Thus there are no more than eight $(-1)$-curves on $Z$, and $K_Z^2 \geqslant 6$. But one cannot find four or more disjoint $\theta$-lines on $X$. So this case is impossible.

If $X$ is minimal (resp. $G$-minimal) and admits a structure of (resp. $G$-equivariant) a conic bundle $\varphi: X \rightarrow B$, then a $\Gal\left( \kka / \ka \right)$-orbit (resp. $G \times \Gal\left( \kka / \ka \right)$-orbit) of each component of singular fibres of $\varphi$ has class $nF$, where $F$ is the class of fibre of $\varphi$. But there are $12$ components of singular fibres of $\varphi$ on $X$ and only eight $\theta$-lines. We have a contradiction. Therefore $\rho(X) = 1$ (resp. $\rho(X)^G = 1$). 

\end{proof}

\begin{remark}
\label{linesGmin}
One can check the following:

\begin{itemize}
\item the set of $\sigma$-lines is $\langle \alpha^3\beta \rangle$-minimal, $\langle \alpha^2\delta \rangle$-minimal, $\langle \alpha\beta\delta\gamma \rangle$-minimal \mbox{and $\langle \delta\gamma, \alpha^2\beta^2\delta\gamma \rangle$-minimal};
\item the set of $\tau$-lines is $\langle \alpha\beta\delta \rangle$-minimal, $\langle \alpha^3\beta \rangle$-minimal, $\langle \alpha^2\delta\gamma \rangle$-minimal \mbox{and $\langle \alpha^3\beta\delta\gamma, \alpha\beta^3\delta\gamma \rangle$-minimal};
\item the set of $\eta$-lines is $\langle \alpha^2\delta \rangle$-minimal, $\langle \alpha\beta\delta \rangle$-minimal, $\langle \alpha^3\beta\gamma \rangle$-minimal \mbox{and $\langle \alpha^2\gamma, \beta^2\gamma \rangle$-minimal}.
\end{itemize}

\end{remark}

\begin{lemma}
\label{linesNmin}

If $\sigma$ (resp. $\tau$, resp. $\eta$) is not a square in $\ka$ then the set of $\sigma$-lines (resp. $\tau$-lines, resp. $\eta$-lines) is $N$-minimal, where $N = \langle \alpha^2\beta^2 \rangle$.
\end{lemma}

\begin{proof}

The two sections $t = \pm \frac{1}{u\sqrt{\sigma}} \left( A(x^2 + y^2) + (A - B) xy \right)$ are permuted by $\Gal\left( \kka / \ka \right)$, and the two sections $t = \pm \frac{1}{u\sqrt{\sigma}} \left( A(x^2 + y^2) + (B - A) xy \right)$ are also permuted by $\Gal\left( \kka / \ka \right)$. Therefore there is an element $h$ in $\Gal\left( \kka / \ka \right)$, whose action coincides with the action of $\gamma$ or $\alpha^2\beta^2\gamma$. In the first case the set of $\sigma$-lines is minimal, and in the second case the set of $\sigma$-lines is $N$-minimal since the action of $\alpha^2\beta^2h$ coincides with the action of $\gamma$.

The proof for $\tau$-lines and $\eta$-lines is similar.

\end{proof}

We will use Corollary \ref{DP2evencrit} to find examples of $\ka$-rational and non-$\ka$-rational quotients~$X / G$. Therefore we want to know fixed points and invariant members of $\mathcal{L}$ for elements in $\langle \alpha\beta, \alpha\beta^3, \delta, \gamma \rangle$. This information is given in Table \ref{Fixedpoints}.

\begin{table}
\caption{} \label{Fixedpoints}

\begin{tabular}{|c|c|c|c|}
\hline
Elements & Order & Invariant members & Fixed points  \\
\hline
$\alpha^3\beta$, $\alpha\beta^3$ & $4$\rule[-5pt]{0pt}{16pt} & $x = 0$, $y = 0$ & $\left(0 : 0 : 1 : \pm q \sqrt{\sigma\tau\eta}\right)$ \\
\hline
$\alpha^3\beta\gamma$, $\alpha\beta^3\gamma$ & $4$\rule[-5pt]{0pt}{16pt} & $x = 0$, $y = 0$ & $\left(0 : 1 : 0 : \pm w \sqrt{\eta}\right)$, $\left(1 : 0 : 0 : \pm w \sqrt{\eta}\right)$ \\
\hline
$\alpha^2\delta$, $\beta^2\delta$ & $4$\rule[-5pt]{0pt}{16pt} & $x = \ii y$, $x = -\ii y$ & $\left(0 : 0 : 1 : \pm q \sqrt{\sigma\tau\eta}\right)$ \\
\hline
$\alpha^2\delta\gamma$, $\beta^2\delta\gamma$ & $4$\rule[-5pt]{0pt}{16pt} & $x = \ii y$, $x = -\ii y$ & $\left(\ii : 1 : 0 : \pm 2v \sqrt{\tau}\right)$, $\left(-\ii : 1 : 0 : \pm 2v \sqrt{\tau}\right)$ \\
\hline
$\alpha\beta\delta$, $\alpha^3\beta^3\delta$ & $4$\rule[-5pt]{0pt}{16pt} & $x = y$, $x = -y$ & $\left(0 : 0 : 1 : \pm q \sqrt{\sigma\tau\eta}\right)$ \\
\hline
$\alpha\beta\delta\gamma$, $\alpha^3\beta^3\delta\gamma$ & $4$\rule[-5pt]{0pt}{16pt} & $x = y$, $x = -y$ & $\left(1 : 1 : 0 : \pm 2u \sqrt{\sigma}\right)$, $\left(-1 : 1 : 0 : \pm 2u \sqrt{\sigma}\right)$ \\
\hline
$\alpha^2\gamma$ & $2$\rule[-8pt]{0pt}{23pt} & $x = 0$, $y = 0$ & $\left(0 : \pm\sqrt{\pm wq{\sqrt{\sigma \tau}}} : w : 0\right)$ \\
\hline
$\beta^2\gamma$ & $2$\rule[-8pt]{0pt}{23pt} & $x = 0$, $y = 0$ & $\left(\pm\sqrt{\pm wq{\sqrt{\sigma \tau}}} : 0 : w : 0\right)$ \\
\hline
$\delta\gamma$ & $2$\rule[-8pt]{0pt}{23pt} & $x = y$, $x = -y$ & $\left(q\tau\eta : q\tau\eta : \pm \sqrt{\pm 2uq\tau\eta\sqrt{\tau\eta}} : 0\right)$ \\
\hline
$\alpha^2\beta^2\delta\gamma$ & $2$\rule[-8pt]{0pt}{23pt} & $x = y$, $x = -y$ & $\left(q\tau\eta : -q\tau\eta : \pm \sqrt{\pm 2uq\tau\eta\sqrt{\tau\eta}} : 0\right)$ \\
\hline
$\alpha^3\beta\delta\gamma$ & $2$\rule[-8pt]{0pt}{23pt} & $x = \ii y$, $x = -\ii y$ & $\left(q\sigma\eta : \ii q\sigma\eta : \pm \sqrt{\pm 2vq\sigma\eta\sqrt{\sigma\eta}} : 0\right)$ \\
\hline
$\alpha\beta^3\delta\gamma$ & $2$\rule[-8pt]{0pt}{23pt} & $x = \ii y$, $x = -\ii y$ & $\left(q\sigma\eta : -\ii q\sigma\eta : \pm \sqrt{\pm 2vq\sigma\eta\sqrt{\sigma\eta}} : 0\right)$ \\
\hline

\end{tabular}

\end{table}

\begin{remark}
\label{allinone}
One can see that for any type of $G$, if $\rho(X)^G = 1$, and $\sigma$, $\tau$, $\eta$, $\sigma \tau$, $\sigma \eta$, $\tau \eta$ and $\sigma \tau \eta$ are not squares in $\ka$, then by Corollary \ref{DP2evencrit} the quotient $X / G$ is not $\ka$-rational. But we want to find examples with stricter conditions on the field $\ka$.
\end{remark}

\begin{lemma}
\label{Quotpoints}

If the surface $X$ given by equation \eqref{exequation} contains a $\ka$-point, then the set of $\ka$-points on~$X / G$ is dense.

\end{lemma}

\begin{proof}
Let~\mbox{$f: X \rightarrow X / N$} be the quotient morphism, and
$$
\pi: \widetilde{X / N} \rightarrow X / N
$$
be the minimal resolution of singularities. By Lemma \ref{DP2type1} the surface $\widetilde{X / N}$ is an Iskovskikh surface.

If $X$ contains a $\ka$-point $p$, that differs from $\left(0 : 0 : 1 : \pm q \sqrt{\sigma\tau\eta}\right)$ then $f(p)$ is a smooth $\ka$-point on $X / N$. Thus the set of $\ka$-points on $\widetilde{X / N}$ is dense by Lemma \ref{IskUnirat}, since $\pi^{-1}f(p)$ is a $\ka$-point on $\widetilde{X / N}$ that does not lie on a section of $\widetilde{X / N} \rightarrow B$ with self-intersection number $-2$. Therefore the set of $\ka$-points on~$X / G \approx \left( \widetilde{X / N} \right) / (G / N)$ is dense.

If $X$ contains $\ka$-points $\left(0 : 0 : 1 : q \sqrt{\sigma\tau\eta}\right)$ then each section of $\widetilde{X / N} \rightarrow B \cong \Pro^1_{\ka}$ with self-intersection number $-2$ is defined over $\ka$. Thus these sections are isomorphic to $B$. Therefore each smooth fibre $F$ over a $\ka$-point on $B$ is isomorphic to $\Pro^1_{\ka}$, and the set of $\ka$-points on $\widetilde{X / N}$ is dense. Therefore the set of $\ka$-points on~$X / G \approx \left( \widetilde{X / N} \right) / (G / N)$ is dense.

\end{proof}

\begin{remark}
\label{ratpoints}
To construct examples of $\ka$-rational and non-$\ka$-rational quotients satisfying the assumptions of Proposition \ref{DP22example} we should find $\ka$-points on $X$. Table \ref{Fixedpoints} contains such points for certain $\sigma$, $\tau$ and $\eta$.

If $\sigma = 1$ then $\left(1 : 1 : 0 : 2u \sqrt{\sigma}\right)$ is a $\ka$-point on $X$.

If $\tau = 1$ then $\left(\ii : 1 : 0 : 2v \sqrt{\tau}\right)$ is a $\ka$-point on $X$.

If $\eta = 1$ then $\left(0 : 1 : 0 : w \sqrt{\eta}\right)$ is a $\ka$-point on $X$.

If $\sigma = \tau = \eta$ then $\left( v \sqrt{\tau} + \ii u\sqrt{\sigma} : w\sqrt{\eta} : 0 : 0 \right)$ is a $\ka$-point on $X$.

If $\eta = \sigma\tau$ then $\left(0 : 0 : 1 : q \sqrt{\sigma\tau\eta}\right)$ is a $\ka$-point on $X$.

In all these cases $X(\ka) \neq \varnothing$ and the set of $\ka$-points on $X / G$ is dense by Lemma \ref{Quotpoints}. These conditions hold for following Examples \ref{example0}--\ref{example8}.
\end{remark}

Now we construct $\ka$-rational and non-$\ka$-rational quotients of $\ka$-rational surfaces $X$. Assume that $q = uvw$. Then the following lemma holds.

\begin{lemma}
\label{linesrat}
If $X$ is given by equation \eqref{exequation}, $X(\ka) \neq \varnothing$ and $q = uvw$ then $X$ is $\ka$-rational.
\end{lemma}

\begin{proof}

All lines on $X$ are defined over $\ka\left( \sqrt{\sigma}, \sqrt{\tau}, \sqrt{\eta} \right)$. Let $\Gamma = \Gal\left( \ka\left( \sqrt{\sigma}, \sqrt{\tau}, \sqrt{\eta} \right) / \ka \right)$. The action of $\Gamma$ on each of the sets of $\sigma$-lines, $\tau$-lines and $\eta$-lines is either trivial or coincides with the action of $\alpha^2\beta^2\gamma$. For any $\sigma$-, $\tau$- or $\eta$-line $E$ one has $\alpha^2\beta^2 E \cdot E = 1$, therefore~$\alpha^2\beta^2\gamma E \cdot E = 0$, since
$$
\left( \alpha^2\beta^2 E + \alpha^2\beta^2\gamma E \right) \cdot E = -K_X \cdot E = 1.
$$
\noindent Let~$f: X \rightarrow Z$ be a $\Gamma$-equivariant contraction of any $\alpha^2\beta^2\gamma$-invariant pair of $\sigma$-lines. Then there are $16$ lines on $Z$, therefore at least one of the lines on $Z$ is the image of a $\tau$- or $\eta$-line $E$. Thus we can $\Gamma$-equivariantly contract the pair $f(E)$ and $f(\alpha^2\beta^2\gamma E)$, and get a surface $W$ with~$K_W^2 \geqslant 6$. This surface is $\ka$-rational by Theorem \ref{ratcrit}. Therefore $X \approx W$ is $\ka$-rational. 

\end{proof}

In following Examples \ref{example0}--\ref{example4} the surface $X$ is $\ka$-rational by Lemma \ref{linesrat}.

\begin{example}
\label{example0}
Assume that $X$ is given by equation \eqref{exequation}, one has
$$
q = uvw, \qquad \sigma = \tau = \eta = 1
$$
\noindent and $G = \langle \alpha^3\beta, \alpha^2\delta \rangle \cong Q_8$. Then $X$ and $X / G$ are $\ka$-rational, and $\rho(X)^G = 1$.

The surface $X / G$ is $\ka$-rational by Corollary \ref{DP2evencrit}, since the isolated fixed points of $N$ are not permuted by $G \times \Gal\left( \kka / \ka \right)$.

The sets of $\sigma$-lines, $\tau$-lines and $\eta$-lines are $G$-minimal by Remark \ref{linesGmin}. Therefore~\mbox{$\rho(X)^G = 1$} by Lemma \ref{linesmincrit}.
 
\end{example}

\begin{example}
\label{example1}
Assume that $\ka$ contains an element $\mu$ such that $\Gal(\ka(\sqrt{\mu}) / \ka) \cong \CG_2$, the surface $X$~is given by equation \eqref{exequation}, one has
$$
q = uvw, \qquad \sigma = \tau = \eta = \mu
$$
\noindent and $G$ is a group of type $(2)$, $(5)$, $(6)$, $(7)$, $(9)$, $(10)$ or $(11)$ of Proposition \ref{DP2}. Then $X$ is $\ka$-rational and $\rho(X)^G = 1$. The quotient $X / G$ is non-$\ka$-rational if $G \cong \CG_2$, $G \cong \CG_4$ or~$G \cong Q_8$, and is $\ka$-rational if $G \cong \CG_2^2$ or $G \cong \DG_8$.

The surface $X / G$ is $\ka$-rational only for $G \cong \CG_2^2$, $G \cong \DG_8$ by Corollary \ref{DP2evencrit}, since the fixed points of $g\in G$ lying on a $g$-invariant member of $\mathcal{L}$ are not contained in one $G \times \Gal\left( \kka / \ka \right)$-orbit, only if $g$ is an element of type $2$ in the notation of Table \ref{table2} (see Table~\ref{Fixedpoints}).

The sets of $\sigma$-lines, $\tau$-lines and $\eta$-lines are $G$-minimal by Lemma \ref{linesNmin}. Therefore~\mbox{$\rho(X)^G = 1$} by Lemma \ref{linesmincrit}.
 
\end{example}

\begin{example}
\label{example2}
Assume that $\ka$ contains an element $\mu$ such that $\Gal(\ka(\sqrt{\mu}) / \ka) \cong \CG_2$, the surface $X$~is given by equation \eqref{exequation}, one has
$$
q = uvw, \qquad \sigma = \tau = \mu, \qquad \eta = 1
$$
\noindent and $G$ is a group $\langle \alpha^2\delta \rangle \cong \CG_4$, $\langle \alpha^3\beta\gamma \rangle \cong \CG_4$ or $\langle \alpha^3\beta\gamma, \alpha^2\delta \rangle \cong Q_8$. Then $X$ and $X / G$ are $\ka$-rational, and $\rho(X)^G = 1$.

The surface $X / \langle \alpha^2\delta \rangle$ is $\ka$-rational by Corollary \ref{DP2evencrit}, since the isolated fixed points of~$N$ are not permuted by $G \times \Gal\left( \kka / \ka \right)$. The surfaces $X / \langle \alpha^3\beta\gamma \rangle$ and $X / \langle \alpha^3\beta\gamma, \alpha^2\delta \rangle$ are $\ka$-rational by Corollary \ref{DP2evencrit}, since the fixed points of $\alpha^3\beta\gamma$ lying on an \mbox{$\alpha^3\beta\gamma$-invariant} member of $\mathcal{L}$ are not contained in one $G \times \Gal\left( \kka / \ka \right)$-orbit (see Table \ref{Fixedpoints}).

The sets of $\sigma$-lines and $\tau$-lines are $G$-minimal by Lemma \ref{linesNmin}, and the set of $\eta$-lines is $G$-minimal by Remark \ref{linesGmin}. Therefore $\rho(X)^G = 1$ by Lemma \ref{linesmincrit}.

\end{example}

\begin{example}
\label{example3}
Assume that $\ka$ contains an element $\mu$ such that $\Gal(\ka(\sqrt{\mu}) / \ka) \cong \CG_2$, the surface $X$~is given by equation \eqref{exequation}, one has
$$
q = uvw, \qquad \sigma = \tau = 1, \qquad \eta = \mu
$$
\noindent and $G  = \langle \alpha^3\beta, \delta\gamma \rangle \cong \DG_8$. Then $X$ is $\ka$-rational, $X / G$ is non-$\ka$-rational, and $\rho(X)^G = 1$.

The surface $X / G$ is non-$\ka$-rational by Corollary \ref{DP2evencrit}, since the fixed points of any element $g \in G$ lying on a $g$-invariant member of $\mathcal{L}$ are contained in one \mbox{$N \times \Gal\left( \kka / \ka \right)$-orbit} (see Table \ref{Fixedpoints}).

The sets of $\sigma$-lines and $\tau$-lines are $G$-minimal by Remark \ref{linesGmin}, and the set of $\eta$-lines is $G$-minimal by Lemma \ref{linesNmin}. Therefore $\rho(X)^G = 1$ by Lemma \ref{linesmincrit}.

\end{example}

\begin{example}
\label{example4}
Assume that $\ka$ contains elements $\mu$ and $\nu$ such that \mbox{$\Gal\left(\ka\left(\sqrt{\mu}, \sqrt{\nu}\right) / \ka\right) \cong \CG^2_2$} and the conic $\mu u^2 + \nu v^2 = \mu\nu w^2$ in $\Pro^2_{\ka}$ has a $\ka$-point, the surface $X$~is given by equation \eqref{exequation}, one has
$$
q = uvw, \qquad \sigma = \mu, \qquad \tau = \nu, \qquad \eta = \mu\nu
$$
\noindent and $G$ is a group $N$ or~$\langle \alpha^2\gamma, \beta^2\gamma \rangle \cong \CG_2^2$. Then $X$ is $\ka$-rational and $\rho(X)^G = 1$. The quotient $X / G$ is $\ka$-rational if~$G \cong \CG_2$ and is non-$\ka$-rational if $G \cong \CG_2^2$.

The surface $X / N$ is $\ka$-rational by Corollary \ref{DP2evencrit}, since the isolated fixed points of $N$ are not permuted by $\Gal\left( \kka / \ka \right)$. The surface $X / \langle \alpha^2\gamma, \beta^2\gamma \rangle$ is non-$\ka$-rational by Corollary~\ref{DP2evencrit}, since the fixed points of any element $g \in G$ lying on a $g$-invariant member of $\mathcal{L}$ are contained in one $G \times \Gal\left( \kka / \ka \right)$-orbit (see Table \ref{Fixedpoints}).

The sets of $\sigma$-lines, $\tau$-lines and $\eta$-lines are $G$-minimal by Lemma \ref{linesNmin}. Therefore~\mbox{$\rho(X)^G = 1$} by Lemma \ref{linesmincrit}.
 
\end{example}

Now we construct $\ka$-rational and non-$\ka$-rational quotients of non-$\ka$-rational surfaces $X$. Assume that $q = uvw\sigma\tau\eta$.

\begin{example}
\label{example5}
Assume that $\ka$ contains an element $\mu$ such that $\Gal(\ka(\sqrt{\mu}) / \ka) \cong \CG_2$, the surface $X$~is given by equation \eqref{exequation}, one has
$$
q = uvw\sigma\tau\eta, \qquad \sigma = \tau = \eta = \mu
$$
\noindent and $G$ is a group of type $(1)$, $(2)$, $(3)$, $(5)$, $(6)$, $(7)$, $(9)$, $(10)$ or $(11)$ of Proposition \ref{DP2}. Then $X$ is non-$\ka$-rational and~$\rho(X)^G = \rho(X) = 1$. The quotient $X / G$ is non-$\ka$-rational if $G$ is trivial, $G \cong \CG_2$, $G \cong \CG_4$ or $G \cong Q_8$, and is $\ka$-rational if $G \cong \CG_2^2$ or $G \cong \DG_8$.

All lines on $X$ are defined over $\ka\left( \mu \right)$. The $\Gal\left( \ka(\mu)/\ka \right)$-orbit of any $(-1)$-curve on $X$ coincides with the $\langle \gamma \rangle$-orbit. Therefore $\rho(X) = 1$ and $X$ is non-$\ka$-rational by Theorem~\ref{ratcrit}.

If $G$ is trivial then $X = X / G$ is non-$\ka$-rational. If $G$ has type $(3)$ then $X / G$ is \mbox{non-$\ka$-rational} by Remark \ref{DP2type2min}. For the other types of $G$ the surface $X / G$ is $\ka$-rational only for $G \cong \CG_2^2$, $G \cong \DG_8$ by Corollary \ref{DP2evencrit}, since the fixed points of $g\in G$ lying on a~$g$-invariant member of~$\mathcal{L}$ are not contained in one $G \times \Gal\left( \kka / \ka \right)$-orbit, only if $g$ is an element of type~$2$ in the notation of Table \ref{table2} (see Table~\ref{Fixedpoints}).

\end{example}

\begin{example}
\label{example6}
Assume that $\ka$ contains an element $\mu$ such that $\Gal(\ka(\sqrt{\mu}) / \ka) \cong \CG_2$, the surface $X$~is given by equation \eqref{exequation}, one has
$$
q = uvw\sigma\tau\eta, \qquad \sigma = \tau = 1, \qquad \eta = \mu
$$
\noindent and $G$ is a group $\langle \alpha^3\beta, \alpha^2\delta\gamma \rangle \cong Q_8$ or $\langle \alpha^3\beta, \delta\gamma \rangle \cong \DG_8$. Then $X$ is non-$\ka$-rational and~$\rho(X)^G = 1$. The quotient $X / G$ is $\ka$-rational if $G \cong Q_8$, and is non-$\ka$-rational if~$G \cong \DG_8$.

All lines on $X$ are defined over $\ka\left( \mu \right)$. The $\Gal\left( \ka(\mu)/\ka \right)$-orbit of any $\sigma$-line or $\tau$-line coincides with the $\langle \alpha^2\beta^2 \rangle$-orbit, and the $\Gal\left( \ka(\mu)/\ka \right)$-orbit of any $\eta$-line coincides with the $\langle \gamma \rangle$-orbit. Therefore a pair of $\Gal\left( \ka(\mu)/\ka \right)$-invariant lines from one of these sets cannot be contracted by Remark~\ref{a2b2lines}. Thus one can contract a pair of $\Gal\left( \ka(\mu)/\ka \right)$-invariant $\theta$-lines and get a minimal del Pezzo surface $Z$ of degree $4$. This surface is non-$\ka$-rational by Theorem \ref{ratcrit}. Therefore $X \approx Z$ is non-$\ka$-rational. 

The surface $X / \langle \alpha^3\beta, \alpha^2\delta\gamma \rangle$ is $\ka$-rational by Corollary \ref{DP2evencrit}, since the fixed points of $\alpha^2\delta\gamma$ lying on an $\alpha^2\delta\gamma$-invariant member of $\mathcal{L}$ are not contained in one $G \times \Gal\left( \kka / \ka \right)$-orbit (see Table \ref{Fixedpoints}). The surface $X / \langle \alpha^3\beta, \delta\gamma \rangle$ is non-$\ka$-rational by Corollary~\ref{DP2evencrit}, since the fixed points of any element~$g \in G$ lying on a $g$-invariant member of $\mathcal{L}$ are contained in one $N \times \Gal\left( \kka / \ka \right)$-orbit (see Table \ref{Fixedpoints}).

The sets of $\sigma$-lines and $\tau$-lines are $G$-minimal by Remark \ref{linesGmin}, and the set of $\eta$-lines is $G$-minimal by Lemma \ref{linesNmin}. Therefore $\rho(X)^G = 1$ by Lemma \ref{linesmincrit}.

\end{example}

\begin{example}
\label{example7}
Assume that $\ka$ contains elements $\mu$ and $\nu$ such that \mbox{$\Gal\left(\ka\left(\sqrt{\mu}, \sqrt{\nu}\right) / \ka\right) \cong \CG^2_2$} and the conic $\mu u^2 + \nu v^2 = w^2$ in $\Pro^2_{\ka}$ has a $\ka$-point (this condition is equivalent to the condition that the conic $\mu u^2 + \nu v^2 = \mu\nu w^2$ in $\Pro^2_{\ka}$ has a $\ka$-point), $X$~is given by equation \eqref{exequation}, one has
$$
q = uvw\sigma\tau\eta, \qquad \sigma = \mu, \qquad \tau = \nu, \qquad \eta = 1
$$
\noindent and $G$ is a group $\langle \alpha^3\beta\gamma \rangle \cong \CG_4$ or $\langle \alpha^2\gamma, \beta^2\gamma \rangle \cong \CG_2^2$. Then $X$ is non-$\ka$-rational \mbox{and $\rho(X)^G = 1$}. The quotient $X / G$ is $\ka$-rational if $G \cong \CG_4$, and is non-$\ka$-rational if~$G \cong \CG_2^2$.

All lines on $X$ are defined over $\ka\left( \mu, \nu \right)$. The $\Gal\left( \ka(\mu, \nu)/\ka \right)$-orbit of any $\sigma$-line or \mbox{$\tau$-line} coincides with the $\langle \alpha^2\beta^2, \gamma \rangle$-orbit, and the $\Gal\left( \ka(\mu, \nu)/\ka \right)$-orbit of any $\eta$-line coincides with the $\langle \alpha^2\beta^2 \rangle$-orbit. Therefore a pair of $\Gal\left( \ka(\mu, \nu)/\ka \right)$-invariant lines from one of these sets can not be contracted by Remark \ref{a2b2lines}. One can check that the \mbox{$\Gal\left( \ka\left(\mu, \nu\right)/\ka \right)$-orbit} of each \mbox{$\theta$-line} consists of four lines. Such quadruple can not be contracted. Therefore $X$ is minimal and non-$\ka$-rational by Theorem \ref{ratcrit}.

The surface $X / \langle \alpha^3\beta\gamma \rangle$ is $\ka$-rational by Corollary \ref{DP2evencrit}, since the fixed points of $\alpha^3\beta\gamma$ lying on an $\alpha^3\beta\gamma$-invariant member of $\mathcal{L}$ are not contained in one $G \times \Gal\left( \kka / \ka \right)$-orbit (see Table \ref{Fixedpoints}). The surface $X / \langle \alpha^2\gamma, \beta^2\gamma \rangle$ is non-$\ka$-rational by Corollary \ref{DP2evencrit}, since the fixed points of any element~$g \in G$ lying on a $g$-invariant member of $\mathcal{L}$ are contained in one $N \times \Gal\left( \kka / \ka \right)$-orbit (see Table \ref{Fixedpoints}).

The sets of $\sigma$-lines and $\tau$-lines are $G$-minimal by Lemma \ref{linesNmin}, and the set of $\eta$-lines is $G$-minimal by Remark \ref{linesGmin}. Therefore $\rho(X)^G = 1$ by Lemma \ref{linesmincrit}.

\end{example}

For the remaining three cases we assume that $q = uvw \sigma\tau$.

\begin{example}
\label{example8}
Assume that $\ka$ contains elements $\mu$ and $\nu$ such that \mbox{$\Gal\left(\ka\left(\sqrt{\mu}, \sqrt{\nu}\right) / \ka \right) \cong \CG^2_2$} and the conic $\mu u^2 + \nu v^2 = \mu\nu w^2$ in $\Pro^2_{\ka}$ has a $\ka$-point, $X$~is given by equation \eqref{exequation}, one has
$$
q = uvw\sigma\tau, \qquad \sigma = \mu, \qquad \tau = \nu, \qquad \eta = \mu\nu
$$
\noindent and $G$ is a group~$N$, $\langle \alpha^3\beta \rangle \cong \CG_4$ or $\langle \alpha^3\beta, \alpha^2\delta \rangle \cong Q_8$. Then $X$ is non-$\ka$-rational, $X / G$ is $\ka$-rational and~$\rho(X)^G = \rho(X) = 1$.

All lines on $X$ are defined over $\ka\left( \mu, \nu \right)$. The $\Gal\left( \ka(\mu, \nu)/\ka \right)$-orbit of any $\sigma$-line or \mbox{$\tau$-line} coincides with the $\langle \alpha^2\beta^2, \gamma \rangle$-orbit, and the $\Gal\left( \ka(\mu, \nu)/\ka \right)$-orbit of any $\eta$-line coincides with the $\langle \gamma \rangle$-orbit. Therefore the sets of these lines are minimal. Hence $\rho(X) = 1$ by Lemma \ref{linesmincrit}, and $X$ is non-$\ka$-rational by Theorem~\ref{ratcrit}.

The surface $X / G$ is $\ka$-rational by Corollary \ref{DP2evencrit}, since the isolated fixed points of $N$ are not permuted by $G \times \Gal\left( \kka / \ka \right)$.

\end{example}

\subsection{The groups $\CG_3$ and $\SG_3$}

In this section we consider del Pezzo surfaces $X$ such that the groups $\CG_3$ and $\SG_3$ of types $(4)$ and $(8)$, listed in Proposition \ref{DP2}, act on $X$. For~$G \cong \CG_3$ and $G 
\cong \SG_3$ we construct examples of $\ka$-rational and non-$\ka$-rational quotients of $\ka$-rational and non-$\ka$-rational del Pezzo surfaces of degree~$2$ such that $\rho(X)^G = 1$.

Assume that the field $\ka$ contains $\omega$. Let $X$ be a del Pezzo surface of degree $2$ given in~$\Pro_{\ka}(1 : 1 : 1 : 2)$ by the equation
\begin{equation}
\label{exequation3}
(x^3 + y^3)z + Ax^2y^2 + 2Bxyz^2 + Cz^4 - t^2 = 0.
\end{equation}

A finite group $G \cong \SG_3$ generated by
$$
(x : y : z : t) \mapsto (\omega x : \omega^2 y : z : t); \qquad  (x : y : z : t) \mapsto (y : x : z : -t)
$$
\noindent acts on $X$. The set of fixed points of the normal subgroup $N \cong \CG_3 \lhd \SG_3$ consists of four isolated fixed points:
$$
p_1 = (1 : 0 : 0 : 0), \quad p_2 = (0 : 1 : 0 : 0), \quad q_1 = \left(0 : 0 : 1 : \sqrt{C} \right), \quad q_2 = \left(0 : 0 : 1 : -\sqrt{C} \right).
$$
Note that $p_1$ and $p_2$ are $\ka$-points on $X$, and by Remark \ref{DP2type4points} the sets of $\ka$-points on $X / N$ and $X / G = \left( X / N \right) / (G / N)$ are dense, since the sections $x = 0$ and $y = 0$ are defined over~$\ka$.

We have the following two corollaries from Remark \ref{DP2type4min}.

\begin{corollary}
\label{DP2C3quotient}
Let $X$ be a surface given by equation \eqref{exequation3} and $N \cong \CG_3$ act on $X$. Then the quotient $X / N$ is $\ka$-rational if and only if $C$ is a square in $\ka$.
\end{corollary}

\begin{proof}
The $N$-fixed points $p_1$ and $p_2$ are defined over $\ka$, therefore only cases $(1)$ and~$(2)$ of Remark \ref{DP2type4min} can be achieved. For the case $(1)$ the quotient $X / N$ is $\ka$-rational and~$\sqrt{C} \in \ka$, for the case $(2)$ the quotient $X / N$ is not $\ka$-rational and~$\sqrt{C} \notin \ka$.
\end{proof}

\begin{corollary}
\label{DP2S3quotient}
Let $X$ be a surface given by equation \eqref{exequation3} and $G \cong \SG_3$ act on $X$. Then the quotient $X / G$ is $\ka$-rational if $C$ is a square in $\ka$. The quotient $X / G$ is not $\ka$-rational if $C$ is not a square in $\ka$, and the roots of the equation
\begin{equation}
\label{S3fixedpoints}
Cz^4 + 2Bx^2z^2 + 2x^3z + Ax^4 = 0
\end{equation}
\noindent are transitively permuted by the Galois group $\Gal\left(\kka / \ka \right)$.
\end{corollary}

\begin{proof}
The element $(x : y : z : t) \mapsto (y : x : z : -t)$ of $G$ permutes the $N$-fixed points $p_1$ and $p_2$, and permutes the $N$-fixed points $q_1$ and $q_2$, therefore only cases $(4)$ and $(5)$ of Remark \ref{DP2type4min} can be achieved.

For the case $(4)$ one has $\sqrt{C} \in \ka$, and the quotient $X / N$ is $G / N$-birationally equivalent to a del Pezzo surface $Z$ of degree $6$. By Corollary \ref{geq5} the quotient $X / G \approx Z / (G / N)$ is $\ka$-rational.

For the case $(5)$ one has $\sqrt{C} \notin \ka$, and the quotient $X / N$ is $G / N$-birationally equivalent to a del Pezzo surface $Z$ of degree $4$ such that $\rho(Z)^{G / N} = 1$. In this case if all fixed points of $G / N$ on $Z$ are transitively permuted by the Galois group, then by \cite[Lemma 6.1]{Tr17} the quotient $X / G \approx Z / (G / N)$ is birationally equivalent to a minimal conic bundle $Y$ such that $K_Y^2 = 2$. Thus $X / G \approx Y$ is not $\ka$-rational by Theorem \ref{ratcrit}.

The fixed points of $G / N$ on $Z$ are transitively permuted by $\Gal\left(\kka / \ka \right)$, if and only if the fixed points of $(x : y : z : t) \mapsto (y : x : z : -t)$ are transitively permuted by~$\Gal\left(\kka / \ka \right)$. These four fixed points on $X$ are given by the equations $t = 0$, $y = x$, and equation~\eqref{S3fixedpoints}.
\end{proof}

Now we want to find conditions on the coefficients of equation \eqref{exequation3} such that $X$ is $\ka$-rational or not. We use the notation of Section $5$. By Lemma \ref{Eltype4} we can assume that the automorphism $(x : y : z : t) \mapsto (\omega x : \omega^2 y : z : t)$ corresponds to the \mbox{element $ab = (123)(456) \in \SG_{7} \subset \mathrm{W}(\mathrm{E}_7)$.} By Proposition \ref{type4center} the image $\Gamma$ of the Galois group~$\Gal\left(\kka / \ka \right)$ in $\mathrm{W}(\mathrm{E}_7)$ is a subgroup of
$$
\langle a, b, cs, r, s, \gamma \rangle \cong \left( \CG_3^2 \rtimes \CG_2 \right) \times \SG_3 \times \CG_2,
$$
\noindent where the elements $a$, $b$, $c$, $r$, $s$ and $\gamma$ are defined in Section $5$.

The element $ab$ has two invariant $(-1)$-curves: $E_7$ and $C_7$. One can easily see that these curves lie in the hyperplane section $z = 0$, and are given by the equations $\sqrt{A} x^2y^2 = \pm t$. Therefore if $A$ is a square in $\ka$ then $E_7$ is $\Gamma$-invariant, and $\Gamma \subset \langle a, b, cs, r, s \rangle$.

\begin{definition}
\label{Eckardtpoint}
A point $p$ on a del Pezzo surface of degree $2$ is called \textit{generalized Eckardt point} if there are four $(-1)$-curves passing through $p$.
\end{definition}

\begin{lemma}
\label{Eckardtpoints}
On the del Pezzo surface $X$ of degree $2$ given by equation \eqref{exequation3} there are six generalized Eckardt points:
$$
\left( 1 : -1 : 0: \pm\sqrt{A} \right), \qquad \left( \omega : -\omega^2 : 0: \pm\sqrt{A} \right), \qquad \left( \omega^2 : -\omega : 0: \pm\sqrt{A} \right).
$$
\end{lemma}

\begin{proof}
Note that the considered points lie in one orbit of the group $G \cong \SG_3$, therefore it is sufficient to show that one of these points is a generalized Eckardt point.

Let us find reducible members of the family of hyperplane sections $z = k(x + y)$ passing through the point $\left( 1 : -1 : 0: \sqrt{A} \right)$:
$$
k(x^3 + y^3)(x + y) + Ax^2y^2 + 2Bk^2xy(x + y)^2 + Ck^4(x + y)^4 - t^2 = 0,
$$
$$
\left( Ck^4 + k \right) (x + y)^4 + \left( 2Bk^2 - 3k \right) xy( x + y)^2 + A(xy)^2 - t^2 = 0.
$$
This section is reducible if $\left( Ck^4 + k \right) (x + y)^4 + \left( 2Bk^2 - 3k \right) xy( x + y)^2 + A(xy)^2$ is a square in $\ka(x , y)$. Therefore one has
$$
\left( 2Bk^2 - 3k \right)^2 - 4A\left( Ck^4 + k \right) = 0, \qquad k\left( 4(B^2 - AC)k^3 - 12Bk^2 + 9k - 4A \right) = 0.  
$$
The last equation has four roots, therefore $\left( 1 : -1 : 0: \sqrt{A} \right)$ is a generalized Eckardt point.
\end{proof}

\begin{remark}
\label{From2to3}
Let $\XX \rightarrow \overline{S}$ be the contraction of the $(-1)$-curve $E_7$. Then the images of the three generalized Eckardt points lying on the $(-1)$-curve $C_7$ are Eckardt points on the cubic surface $\overline{S}$, permuted by the group $N = \langle ab \rangle \cong \CG_3$. In this case we can apply results of \cite[Lemma 5.7]{Tr16b} and see that $\Gamma$ is a subgroup of $\langle a^2b, cs, r, s, \gamma \rangle$. The elements $r$, $s$ and $\gamma$ trivially act on the roots of the equation
\begin{equation}
\label{Equationabc}
4(B^2 - AC)k^3 - 12Bk^2 + 9k - 4A = 0,
\end{equation}
\noindent and the elements $a^2b$ and $cs$ permute three and two roots of this equation respectively.
\end{remark}

Note that there are six $N$-invariant conic bundle structures on $\XX$ that have the following classes of fibres:
$$
L - E_7, \qquad 2L - \sum \limits_{i = 1}^3 E_i - E_7, \qquad 4L - 2\sum \limits_{i = 1}^3 E_i - \sum \limits_{i = 4}^6 E_i - E_7,
$$
$$
5L - 2\sum \limits_{i = 1}^6 E_i - E_7, \qquad 4L - \sum \limits_{i = 1}^3 E_i - 2\sum \limits_{i = 4}^6 E_i - E_7, \qquad 2L - \sum \limits_{i = 4}^6 E_i - E_7.
$$
Let $F$ be the class of a fibre of one of these conic bundles. The group $N$ has only two invariant $(-1)$-curves, therefore it permutes singular fibres of $\varphi_{|F|}: \XX \rightarrow \Pro^1_{\kka}$. Thus there are exactly two \mbox{$N$-invariant curves} in the linear system $|F|$. Note \mbox{that $\gamma F = cs F = -2K_X - F$}. Therefore there are six reducible $\langle N, \gamma \rangle$-invariant curves with classes $-2K_X$. The elements $a^2b$, $cs$ and $\gamma$ do not permute these curves, and the elements $r$ and $s$ nontrivially permute this set of curves.

An $\langle N, \gamma \rangle$-invariant curve with class $-2K_X$ is given by $yz = \lambda x^2$, $xz = \lambda y^2$, $xy = \lambda z^2$ or $t = 0$. Let us find reducible curves given by $yz = \lambda x^2$:
$$
\lambda(x^3 + y^3)x^2y^3 + Ax^2y^6 + 2B\lambda^2x^5y^3 + C\lambda^4x^8 - y^4t^2 = 0,
$$
$$
x^2\left( C\lambda^4x^6 + (\lambda + 2B\lambda^2)x^3y^3 + (A + \lambda)y^6 \right) - y^4t^2 = 0.
$$
This section is reducible if $C\lambda^4x^6 + (\lambda + 2B\lambda^2)x^3y^3 + (A + \lambda)y^6$ is a square in $\ka(x , y)$. Therefore one has
$$
(\lambda + 2B\lambda^2)^2 - 4(A + \lambda)C\lambda^4 = 0, \qquad \lambda^2\left(4C\lambda^3 - 4(B^2 - AC)\lambda^2 - 4B\lambda - 1 \right) = 0.
$$
For the root $\lambda = 0$ the section $yz = 0$ consists of three irreducible components, and it is not our case. Therefore three curves corresponding to $N$-invariant fibres are given by~$yz = \lambda x^2$, where
\begin{equation}
\label{Equationrs}
4C\lambda^3 - 4(B^2 - AC)\lambda^2 - 4B\lambda - 1 = 0.
\end{equation}
\noindent The three other curves are given by $xz = \lambda y^2$ for the same $\lambda$. One can check that any curve given by $xy = \lambda z^2$, $\lambda \neq 0$ is irreducible.

Now construction of examples is reduced to finding coefficients $A$, $B$ and $C$ such that the Galois groups of \eqref{Equationabc}, \eqref{Equationrs} and \eqref{S3fixedpoints} satisfy certain conditions. For simplicitly we put $B = 0$. The following lemma follows from direct computations.

\begin{lemma}
\label{discrimination}
Let $B = 0$ then the order of the Galois group of \eqref{Equationabc} is divisible by $2$ if and only if $AC(16A^3C - 27)$ is not a square in $\ka$, and the order of the Galois group of~\eqref{Equationrs} is divisible by $2$ if and only if $16A^3C - 27$ is not a square in $\ka$.
\end{lemma}

\begin{example}
\label{nrattonrat}
Assume that a del Pezzo surface $X$ of degree $2$ is given by equation~\eqref{exequation3}, and $A$ is not a square in $\ka$, $B = 0$, $C = \frac{9}{4A^3}$. Then $16A^3C - 27 = 9$ \mbox{and $AC(16A^3C - 27) = \frac{9}{A^2}$} are squares in $\ka$. Therefore by Lemma \ref{discrimination} the only element of order $2$ in $\Gamma$ is $\gamma$. Thus $X$ is not $\ka$-rational.

The number $C$ is not a square in $\ka$. Therefore $X / N$ is not $\ka$-rational by Corollary \ref{DP2C3quotient}.

Equation \eqref{S3fixedpoints} takes form
$$
\frac{9}{4A^3} \cdot z^4 + 2x^3z + Ax^4 = 0.
$$
One can put $Ax = 1$, and get the equation $9z^4 + 8z + 4 = 0$. If the roots of this equation are transitively permuted by the Galois group then the quotient $X / G$ is not $\ka$-rational by Corollary \ref{DP2S3quotient}. For example, this holds for $\ka = \mathbb{Q} (\omega)$.
\end{example}

\begin{example}
\label{nrattorat}
Assume that a del Pezzo surface $X$ of degree $2$ is given by equation~\eqref{exequation3}, and $A$ is not a square in $\ka$, $B = 0$, $C = \frac{u^2}{A^2}$, and $u$ is a nonzero number such that~\mbox{$16Au^2 - 27 = Aw^2$} (such $u$ and $w$ always exist). Then $16A^3C - 27 = Aw^2$ is not a square in $\ka$, and~$AC(16A^3C - 27) = \frac{u^2w^2}{A^2}$ is a square $\ka$. Therefore by Lemma~\ref{discrimination} the group $\Gamma$ contains an element conjugate to $s\gamma$. Thus $X$ is not $\ka$-rational by Lemma~\ref{type4min2nonrational} and Theorem \ref{ratcrit}.

The number $C$ is a square in $\ka$. Therefore $X / N$ and $X / G$ are $\ka$-rational by Corollaries~\ref{DP2C3quotient} and \ref{DP2S3quotient}.
\end{example}

Now we construct examples of $\ka$-rational and non-$\ka$-rational quotients of $\ka$-rational del Pezzo surfaces $X$ of degree $2$ such that $\rho(X)^N = 1$. According to Proposition \ref{type4minrational} in this case the group $\Gamma$ is conjugate to $\langle r, cs\gamma \rangle$, $\langle r, c\gamma \rangle$ or $\langle r, c\gamma, s \rangle$. In particular, this group does not contain the elements $a^2b$, $a^2br$ and $a^2br^2$. It means that equation \eqref{Equationabc} has a root defined over~$\ka$. We assume that this root is $k = 1$. Then one has
$$
C = \frac{9 - 4A}{4A}, \qquad 16A^3C - 27 = 4A^2(9 - 4A) - 27 = (2A - 3)^2(-4A - 3),
$$
$$
AC(16A^3C - 27) = \frac{(2A - 3)^2(9 - 4A)(-4A - 3)}{4}.
$$

\begin{example}
\label{rattorat}
Assume that a del Pezzo surface $X$ of degree $2$ is given by equation~\eqref{exequation3}, and $A = -1$, $B = 0$, $C = - \frac{13}{4}$. Then $9 - 4A = 13$, and $-4A - 3 = 1$ is a square. Equation~\eqref{Equationrs} takes form
$$
-13\lambda^3 + 13\lambda^2 - 1 = 0.
$$
Assume that this equation does not have a root in $\ka$, $\sqrt{-13} \in \ka$ and $\sqrt{13} \notin \ka$ (this holds, for example, for $\ka = \mathbb{Q}\left(\omega, \sqrt{-13}\right)$). Then $\Gamma$ contains $r$, and by Lemma \ref{discrimination} one has~$\Gamma = \langle r, cs\gamma \rangle$. Thus $X$ is $\ka$-rational by Proposition \ref{type4minrational} and Theorem \ref{ratcrit}.

The number $C = - \frac{13}{4}$ is a square in $\ka$. Therefore $X / N$ and $X / G$ are $\ka$-rational by Corollaries~\ref{DP2C3quotient} and \ref{DP2S3quotient}.
\end{example}

\begin{example}
\label{rattonrat}
Assume that a del Pezzo surface $X$ of degree $2$ is given by equation~\eqref{exequation3}, and $A = 2$, $B = 0$, $C = \frac{1}{8}$. Then $9 - 4A = 1$ is a square, and $-4A - 3 = -11$. Equation~\eqref{Equationrs} takes form
$$
\lambda^3 + 2\lambda^2 - 2 = 0.
$$
Assume that this equation does not have a root in $\ka$, $\sqrt{-22} \in \ka$ and $\sqrt{2} \notin \ka$ (this holds, for example, for $\ka = \mathbb{Q}\left(\omega, \sqrt{-22}\right)$). Then $\Gamma$ contains $r$, and by Lemma \ref{discrimination} one has~$\Gamma = \langle r, c\gamma \rangle$. Thus $X$ is $\ka$-rational by Proposition \ref{type4minrational} and Theorem \ref{ratcrit}.

The number $C = \frac{1}{8}$ is not a square in $\ka$. Therefore $X / N$ is not $\ka$-rational by Corollary~\ref{DP2C3quotient}.

Equation \eqref{S3fixedpoints} takes form
$$
z^4 + 16x^3z + 16x^4 = 0.
$$
If the roots of this equation are transitively permuted by the Galois group then the quotient $X / G$ is not $\ka$-rational by Corollary \ref{DP2S3quotient}. For example, this holds for~\mbox{$\ka = \mathbb{Q} (\omega, \sqrt{-22})$}.
\end{example}

\bibliographystyle{alpha}

\begin{thebibliography}{XX}

\bibitem[ATLAS]{ATLAS}
J.\,Conway, R.\,Curtis, S.\,Norton, R.\,Parker, R.\,Wilson,
\newblock Atlas of finite groups.
\newblock Clarendon Press, Oxford, 1985

\bibitem[CT88]{CT88}
D.\,F.\,Coray, M.\,A.\,Tsfasman,
\newblock Arithmetic on singular Del Pezzo surfaces,
\newblock Proc. Lond. Math. Soc., 1988, 57, 25--87

\bibitem[Dol12]{Dol12}
I.\,V.\,Dolgachev,
\newblock Classical algebraic geometry: a modern view,
\newblock Cambridge University Press, Cambridge, 2012

\bibitem[DI09]{DI1}
I.\,V.\,Dolgachev, V.\,A.\,Iskovskikh,
\newblock Finite subgroups of the plane Cremona group,
\newblock In: Algebra, arithmetic, and geometry, vol. I: In Honor of Yu.\,I.\,Manin, Progr. Math., 269, 443--548, Birkh\"auser, Basel, 2009

\bibitem[Elk99]{Elk99}
N.\,D.\,Elkies,
\newblock The Klein quartic in number theory,
\newblock In: Levy, S. (ed.) The Eightfold Way: The Beauty of Klein's Quartic Curve, pp. 51--102. Cambridge Univ. Press, Cambridge (1999)

\bibitem[Isk71]{Isk71}
V.\,A.\,Iskovskikh,
\newblock A counterexample to the Hasse principle for a system of two quadratic forms in five variables,
\newblock Mat. Zametki, 1971, 10, 253--257 (in Russian); translation in Math. Notes 1971, 10, 575--577

\bibitem[Isk79]{Isk79}
V.\,A.\,Iskovskikh,
\newblock Minimal models of rational surfaces over arbitrary field,
\newblock Math. USSR Izv., 1979, 43, 19--43 (in Russian); translation in Math. USSR Izv., 1980, 14, no. 1, 17--39

\bibitem[Isk96]{Isk96}
V.\,A.\,Iskovskikh,
\newblock Factorization of birational mappings of rational surfaces from the point of view of Mori theory,
\newblock Uspekhi Mat. Nauk, 1996, 51, 3--72 (in Russian); translation in Russian Math. Surveys, 1996, 51, 585--652

\bibitem[Man67]{Man67}
Yu.\,I.\,Manin,
\newblock Rational surfaces over perfect fields. II,
\newblock Math. USSR-Sbornik, 1967, 1, 141--168 (in Russian)

\bibitem[Man74]{Man74}
Yu.\,I.\,Manin.
\newblock Cubic forms: algebra, geometry, arithmetic,
\newblock In: North-Holland Mathematical Library, Vol. 4, North-Holland Publishing Co., Amsterdam-London; American Elsevier Publishing Co., New York, 1974

\bibitem[STV14]{STV14}
C.\,Salgado, D.\,Testa, A.\,Varilly-Alvarado,
\newblock On the unirationality of del Pezzo surfaces of degree two,
\newblock J. London Math. Soc., 2014, 90, 121--139

\bibitem[Tr16a]{Tr16a}
A.\,Trepalin,
\newblock Quotients of conic bundles,
\newblock Transformation Groups, 2016, 21(1), 275--295

\bibitem[Tr16b]{Tr16b}
A.\,Trepalin,
\newblock Quotients of cubic surfaces,
\newblock European Journal of Mathematics, 2016, 2, 333--359

\bibitem[Tr17]{Tr17}
A.\,Trepalin,
\newblock Quotients of del Pezzo surfaces of high degree,
\newblock preprint, see http://arxiv.org/abs/1312.6904, to appear in Transactions of the American Mathematical Society


\end{thebibliography}

\end{document}